\RequirePackage[l2tabu, orthodox]{nag}

\documentclass[11pt]{amsart}
\usepackage{fullpage,url,amssymb,enumerate,colonequals,amsmath}

\usepackage{mathrsfs} % for \mathscr (script letters)
\usepackage{tikz-cd} % for commutative diagrams
\usepackage{comment} % for comment
\usepackage{tikz}
\usepackage{tabularx}
\usepackage{graphicx}
\usepackage{amsfonts}

\usepackage[OT2,T1]{fontenc}
\DeclareSymbolFont{cyrletters}{OT2}{wncyr}{m}{n}
\DeclareMathSymbol{\Sha}{\mathalpha}{cyrletters}{"58}

\newtheorem{thm}{Theorem}
\newtheorem{prop}{Proposition}

\theoremstyle{definition}
\newtheorem{defn}{Definition}[section]
\newtheorem{lem}{Lemma}[section] 
\newtheorem{rmk}{Remark} 

\newtheorem{conj}{Conjecture}

\newcommand{\bbC}{\mathbb{C}}

\newcommand{\bbP}{\mathbb{P}}

\newcommand{\bbF}{\mathbb{F}}
\newcommand{\bbZ}{\mathbb{Z}}

\newcommand{\cA}{\mathcal{A}}
\newcommand{\cC}{\mathcal{C}}

\newcommand{\cE}{\mathcal{E}}
\newcommand{\cF}{\mathcal{F}}

\newcommand{\cL}{\mathcal{L}}
\newcommand{\cM}{\mathcal{M}}

\newcommand{\cO}{\mathcal{O}}

\newcommand{\cR}{\mathcal{R}}
\newcommand{\cS}{\mathcal{S}}
\newcommand{\cT}{\mathcal{T}}
\newcommand{\cU}{\mathcal{U}}
\newcommand{\cV}{\mathcal{V}}
\newcommand{\cW}{\mathcal{W}}
\newcommand{\cX}{\mathcal{X}}
\newcommand{\cZ}{\mathcal{Z}}
\newcommand{\NS}{\mathrm{NS}}

\newcommand{\Fbar}{{\overline{\bbF}}}
\newcommand{\taub}{\overline{\tau}}

\DeclareMathOperator{\Char}{char}
\DeclareMathOperator{\Cl}{Cl}

\DeclareMathOperator{\im}{im}

\DeclareMathOperator{\Pic}{Pic}

\DeclareMathOperator{\Spec}{Spec}

\DeclareMathOperator{\supp}{supp}

\makeatletter
\g@addto@macro\bfseries{\boldmath} % This makes math in section titles bold.
\makeatother

\usepackage{microtype}

\usepackage[colorlinks=true,urlcolor=blue, linkcolor=red]{hyperref}

\usepackage[alphabetic,backrefs,lite,nobysame]{amsrefs} % for bibliography

\begin{document}
\author{Yi Wei}
\address{University of Wisconsin-Madison, Department of Mathematics, 480 Lincoln Dr
\hfill \newline\texttt{}
 \indent WI 53706, USA} 
\email{{ wei83@wisc.edu}}
 \urladdr{\url{https://people.math.wisc.edu/~wei83/}}
\title{Generic Green's Conjecture and the Geometric Syzygy Conjecture in Positive Characteristic}
\date{}

% Abstract
\begin{abstract}
    We study the syzygies of canonical curves of genus $g\geq 3$ over an algebraically closed field $\bbF$ of characteristic $p>0$. We provide a new proof of generic Green's Conjecture for $p\geq\frac{g+4}{2}$. Using the techniques from the even-genus case, we establish a significant case of the Geometric Syzygy Conjecture for the last syzygy space of a general even-genus canonical curve (assuming $p>g$). In characteristic 0, it was shown in prior work that this case implies the full conjecture.
    %We introduce a universal Lazarsfeld-Mukai bundle for a smooth family of K3 surfaces.
    %We use the construction on a smooth versal $\bbF$-deformation of a primitively polarized K3 surface of even genus . 
\end{abstract}
\maketitle
% Introduction
\section{Introduction}\label{Intro}

Let $C$ be a smooth, non-hyperelliptic projective curve of genus $g\geq 3$. Its canonical bundle induces a projectively normal embedding $C\hookrightarrow\bbP^{g-1}$ by Noether's theorem. In 1984, Green formulated a conjecture (\cite{GM84}) relating the intrinsic complexity of a smooth curve to the syzygies of its canonical embedding. More precisely, let $K_{i,j}(C,\omega_C)$ denote the $(i,j)$-th Koszul cohomology group of the canonical bundle $\omega_C$, and let $\mathrm{Cliff}(C)$ be the Clifford index of $C$, then
$$K_{i,2}(C,\omega_C)=0,\quad  \forall i\leq i' \Longleftrightarrow \mathrm{Cliff}(C)>i'.$$
Equivalently, $K_{i,1}(C,\omega_C)=0$ if and only if $i\geq g-\mathrm{Cliff}(C)-1$ by the duality $K_{i,1}(C,\omega_C)\simeq K_{g-i-2,2}(C,\omega_C)^{\vee}$. The conjecture generalizes Petri's theorem, which states that if $\mathrm{Cliff}(C)\geq 2$ then the ideal $I_{C/\bbP^{g-1}}$ of a canonical curve is generated by quadrics. Furthermore, Andreotti–Mayer \cite{AM67} and Arbarello–Harris \cite{AH81} proved that those quadrics can be chosen to have rank at most 4, provided the curve is general. 

Although Green's Conjecture remains largely open to this day, Voisin's work \cites{VC02, VC05}, in the early 2000s resolved the conjecture for general curves, called the \textit{generic Green's conjecture}, over complex numbers $\bbC$. The key idea is the intricate study of Hilbert schemes of polarized K3 surfaces $(X,L)$ when smooth hyperplane sections yield smooth canonical curves $(C,\omega_C)$. In this case, the hyperplane restriction theorem \cite{GM84} sets the isomorphism between Koszul cohomology groups: $K_{i,j}(C,\omega_C)=K_{i,j}(X,L)$, for all $i,j$. Recent work by Kemeny \cite{KM20} largely simplified Voisin's proof. Another breakthrough towards generic Green's conjecture is established by Aprodu-Farkas-Papadima-Raicu-Weyman \cite{AFPRW}, whose proof also extends to positive characteristic. However, Green's Conjecture in positive characteristic fails for small prime $p$ (See \cite{SF86} and \cite{MS10}). Therefore, a lower bound for characteristic $p$ is necessary. \smallskip

From now on, let $\bbF$ be any algebraically closed field in characteristic $p>0$. One of our main theorems settles generic Green's conjecture in positive characteristic. 

\begin{thm}[Generic Green's Conjecture in char $p$]\label{GreenConjecture}
Let $\bbF$ be any algebraically closed field with $\Char(\bbF)=p>0$. Let $C$ be a general smooth curve over $\bbF$ of even genus $g=2k$ or odd genus $g=2k+1$ for $k\geq 2$. Assume $p\geq \frac{g+4}{2}$. Then $K_{k,1}(C,\omega_C)=0$. In other words, $C$ satisfies generic Green's Conjecture.
\end{thm}

The above result first appeared in \cite{AFPRW}*{Thm 1.2}, under a milder assumption $p\geq\frac{g+2}{2}$. See also Raicu-Sam \cite{RS21} for an improved sharp bound and a new proof by Park \cite{PJ22}. Their proofs employed entirely different and quite technical methods. Our approach is conceptually simpler and uses the geometry of K3 surfaces, under the framework of Kemeny \cite{KM20}. \smallskip

Schreyer formulated a refined version of the above results, known as the \textit{Geometric Syzygy Conjecture}, which provides a more detailed structure on syzygies of canonical curves. It generalizes the classical theorem of Andreotti–Mayer–Arbarello–Harris concerning the ranks of quadrics in the canonical ideal.  In essence, Schreyer’s conjecture predicts that each linear syzygy space $K_{i,1}(C,\omega_C)$ admits a distinguished basis consisting of syzygies of minimal rank, where the notion of \textit{rank} for an individual syzygy is well defined.

\begin{conj}[The Geometric Syzygy Conjecture]
For a general curve of genus $g$, all linear
syzygy spaces $K_{i,1}(C,\omega_C)$ are spanned by syzygies of minimal rank $i+1$.
\end{conj}

In this paper, we prove the conjecture in positive characteristic for the last linear syzygy space of a general curve with even genus. This is a new partial result toward the conjecture.

\begin{thm}[The Geometric Syzygy Conjecture for Even Genus] \label{Geometric}
Fix an algebraically closed field $\bbF$ with $\Char(\bbF)=p>0$. Let $C$ be a general smooth curve over $\bbF$ of even genus $g=2k\geq 4$. Assume $p>g=2k$. Then $K_{k-1,1}(C,\omega_C)$ is generated by syzygies of lowest possible rank $k$. More precisely, $K_{k-1,1}(C,\omega_C)$ is generated by the rank $k$ syzygies
$$\alpha\in K_{k-1,1}\big(C,\omega_C;H^0(\omega_C\otimes A^{-1})\big),\quad A\in W_{k+1}^1(C).$$
\end{thm}
Over complex numbers $\bbC$, Theorem \ref{Geometric} is proved in \cite{KM20}. Furthermore, it was shown that this case implies the full Geometric Syzygy Conjecture in \cite{KM24} in a sophisticated way. One may hope the method in this paper can be extended to prove the full conjecture in positive characteristic.\smallskip

Both Theorem \ref{GreenConjecture} and Theorem \ref{Geometric} follow from similar results on polarized K3 surfaces. We call a pair $(X,L)$, where $X$ is a K3 surface and $L$ is a line bundle on $X$, a polarized K3 surface of degree $2d$ if $L$ is ample and $L^2 = 2d$. The pair $(X,L)$ is said to be primitively polarized if, in addition, $L$ is primitive. Over any algebraically closed field, there exist primitively polarized K3 surfaces $(X,L)$ of degree $2g-2$ for any $g\geq 3$. (see \cite{HD16}*{\S II. Thm 4.2})\smallskip

Wr treat the even-genus and odd-genus cases separately. The remainder of the Introduction focuses on the even-genus case, while the odd-genus case is addressed in Section \ref{Oddgenus}. \smallskip   

Let $(X,L)$ be any primitively polarized K3 surface over $\bbF$ of even genus $g=2k \geq 4$, equivalently $L^2=2g-2$. A general member $C\in|L|$ is irreducible and smooth of genus $g$ \cite{HD16}*{\S II. Cor 3.6}. We first get a result of Brill-Noether theory for such a curve $C$. 

\begin{prop}\label{Save}
Fix an algebraically closed field $\bbF$ with $\Char(\bbF)\neq 2$. Let $(X,L)$ be a general member of any component of the moduli space of primitively polarized K3 surfaces over $\bbF$ of genus $g$ for $g\geq 3$. Let $r,d$ be such integers that the Brill-Noether number $\rho(r,d,g)=0$. For a general smooth curve $C\in|L|$, consider the Brill-Noether locus $W_d^r(C)$. Then each $A\in W_d^r(C)$ has $h^0(A)=r+1$ and $A$ is base-point free.
\end{prop}
Applying the above result to the case when $r=1, d=k+1, g=2k$, for a general curve $C\in |L|$,  we obtain a base-point free line bundle $A\in W_{k+1}^1(C)$ with $h^0(A)=2$. To any such pair $(C,A)$, one associates a rank 2 \textit{Lazarsfeld-Mukai bundle} $E=E_{C,A}$ on $X$. The dual bundle $E^{\vee}$ fits into the exact sequence
$$0\to E^{\vee}\to H^0(C,A)\otimes \cO_X\to i_*A\to 0,$$
where $i:C\hookrightarrow X$ is the inclusion. The vector bundle $E$ has numerical invariants 
$c_1(E)=\det(E)=L, c_2(E)=k+1; h^0(E)=k+2, h^1(E)=h^2(E)=0, h^0(E^{\vee})=h^1(E^{\vee})=0$ \cite{LR86}. A general section $s\in H^0(E)$ gives a zero-dimensional locus $Z(s)$ on a general curve $C\in|L|$ with $\deg Z(s)=k+1$. Moreover, every $g^1_{k+1}$ on curves arises as a section of $E$. One of our key results is the following

\begin{thm}\label{Main}
Fix an algebraically closed field $\bbF$ with $\mathrm{char}(\bbF)=p>0$. Let $(X,L)$ be a general member in any component of the moduli space of primitively polarized K3 surfaces of even genus $g=2k\geq 4$ over $\bbF$. Assume $p\geq k+2=\frac{g+4}{2}$. Let $E$ be the rank two Lazarsfeld-Mukai bundle defined above. For any nonzero $s\in H^0(E)$, the space $K_{k-1,1}(X,L;H^0(L\otimes I_{Z(s)}))$ is a one-dimensional subspace of $K_{k-1,1}(X,L)$. Then the natural morphism
\[\begin{array}{crcl}
\widetilde{\psi}:&\bbP\big(H^0(E)\big)&\to&\bbP\big(K_{k-1,1}(X,L)\big)\\
&[s]&\mapsto&\big[K_{k-1,1}\big(X,L;H^0(L\otimes I_{Z(s)})\big)\big] 
\end{array}\] 
is the Veronese embedding of degree $k-2$. 

In particular, $\psi$ induces a natural isomorphism $\mathrm{Sym}^{k-2}H^0(X,E)\simeq K_{k-1,1}(X,L)$.
\end{thm}
The above result over $\bbC$ is proved in \cite{KM20}*{Thm 2}, where K3 surface $(X,L)$ is assumed to have $\Pic(X)=\bbZ[L]$, indeed a general case in characteristic zero. Kemeny's proof makes essential use of this property, which poses a serious challenge in positive characteristic. In particular, the Tate conjecture implies that any K3 surface over $\Fbar_p$ has even Picard number \cite{HD16}*{\S XVII. Cor 2.9}. \smallskip

Theorem \ref{Main} implies $K_{k,1}(C,\omega_C)=K_{k,1}(X,L)=0$ by a dimension count (see Section \ref{Greenconjecture}). In practice, however, we first establish this vanishing directly and then deduce the theorem (see the proof of Proposition \ref{Kem}). It is well known that this single vanishing is sufficient to verify generic Green's conjecture. \smallskip

In Section \ref{Evengenus}, we show that under the assumptions of the theorem, any nonzero section $s\in H^0(E)$ is regular, that is, its zero locus $Z(s)\subseteq X$ is zero-dimensional, and $h^0(X,L\otimes I_{Z(s)})=k+1$. If we restrict the subspaces $K_{k-1,1}(X,L;H^0(L\otimes I_{Z(s)}))\subseteq K_{k-1,1}(X,L)$ to a curve $C\in|L|$ containing the locus $Z(s)$, then $K_{k-1,1}(X,L;H^0(L\otimes I_{Z(s)}))$ restricts to
    $$K_{k-1,1}(C,\omega_C;H^0(\omega_C\otimes A^{-1}))\subseteq K_{k-1,1}(C,\omega_C),$$
    where $A\in W_{k+1}^1(C)$ is the line bundle associated with the divisor $Z(s)\subseteq C$. In particular, each syzygy $\alpha\in K_{k-1,1}(X,L;H^0(L\otimes I_{Z(s)}))$ has rank $k+1$, but drops rank by one upon restriction to $C$. This is a crucial step toward Theorem \ref{Geometric}. \smallskip

\noindent\textbf{Proof Strategy.} The proof is inspired by \cite{KM20}. A key contribution of this paper is the use of deformation arguments on moduli spaces of K3 surfaces to modify the assumptions in \cite{KM20}. \smallskip
 
Theorem \ref{Main} is proved by two steps. We first weaken the assumption in \cite{KM20}*{Thm 2} as follows.

\begin{prop}\label{Kem}
Let $(X,L)$ be a primitively polarized K3 surface of even genus $g=2k$ for $k\geq 2$ over an algebraically closed field $\bbF$. Assume $\Char(\bbF)=p\geq \frac{g+4}{2}=k+2$. Suppose that for a smooth curve $C\in|L|$, there exists a line bundle $A\in W_{k+1}^1(C)$ with $h^0(A)=2$ and $A$ base-point free. Let $E$ be a rank 2 Lazarsfeld-Mukai bundle associated with such a pair $(C,A)$. Under the key assumption: 
\begin{equation*}\label{key-assumption}
    \textrm{any nonzero $s\in H^0(E)$ is regular, i.e. the vanishing locus $Z(s)$ is zero-dimensional} \tag{\mbox{$\dagger$}}
\end{equation*} 
the conclusion of Theorem \ref{Main} holds.
\end{prop}
Let $Z\subseteq X\times\bbP(H^0(E))$ be the reduced closed subvariety defined by $\{(x,s)\,|\,s(x)=0\}$. If every nonzero section $s\in H^0(E)$ is regular, then $Z$ is locally complete intersection and hence Cohen-Macaulay. By the Miracle Flatness Theorem, it follows that
    \begin{equation*}\label{key-assumption-2}
        \textrm{the second projection $Z\to\bbP\big(H^0(E)\big)$ is finite and flat.} \tag{\mbox{$\dagger\dagger$}}
    \end{equation*}
    Conversely, condition (\ref{key-assumption-2}) clearly implies (\ref{key-assumption}). \smallskip

\begin{rmk}\label{Regular}
If $\Pic(X)=\bbZ[L]$, then every nonzero section $s\in H^0(E)$ is regular. Indeed, suppose $Z(s)$ is one-dimensional for some nonzero $s$. Then $H^0\big(E(-D)\big)\neq 0$ for some effective divisor $D$, which must be a multiple of $L$, that is, $E(-D)=E(-mL)$ for some positive $m$. Since $H^0\big(E(-mL)\big)\subseteq H^0\big(E(-L)\big)$ and $E(-L)\simeq E^{\vee}$ with $h^0(E^{\vee})=0$, this yields a contradiction.
\end{rmk}

Secondly, we apply a deformation argument to conclude that a general K3 surface over $\bbF$ satisfies the key assumption (\ref{key-assumption}), equivalently (\ref{key-assumption-2}). \smallskip

Let $(X,L)$ be an arbitrary primitively polarized K3 surface over $\bbF$ of degree $2g-2$. Consider a versal $\bbF$-deformation $(\cX,\cL)\to\cT$ of $(X,L)$, and construct the associated family of Lazarsfeld-Mukai bundles $\cE$ as described in Section \ref{Uni}. We show that for a general closed point $t\in\cT$, the fiber $(X_t,\cL_t)$ satisfies condition (\ref{key-assumption-2}). By the generic flatness theorem and flat base change, it suffices to prove the case for the geometric generic fiber $(X_{\taub},\cL_{\taub})$. Note that by Ogus' Theorem \cite{OA79}*{Thm 2.9}, the geometric generic fiber $(X_{\taub},\cL_{\taub})$ has Picard number one. It then satisfies the key assumption by Remark \ref{Regular}. \smallskip

The remaining odd genus case $g=2k+1$ of Green's Conjecture is addressed in Section \ref{Oddgenus} using a similar deformation-theoretic approach. We consider a versal $\bbF$-deformation space of a primitively polarized K3 surface $(X;L,R)$ with intersection lattice $L^2=2g-2, R^2=-2, L\cdot R=2$. We show the geometric generic K3 fiber satisfies the vanishing $K_{k,1}(X_{\taub},\cL_{\taub})=0$. By the semi-continuity of Koszul cohomology, the same vanishing holds for a general closed fiber $(X_t,\cL_t)$ for $t\in\cT$. This concludes the proof of generic Green's conjecture in the odd-genus case. \smallskip

\noindent\textbf{Summary.} The paper is organized as follows. Section \ref{Prelim} provides the necessary background on the relevant conjectures and establishes Proposition \ref{Kem}. Section \ref{Uni} constructs a family of Lazarsfeld-Mukai bundles and proves Proposition \ref{Save}. Section \ref{Evengenus} verifies the key assumption (\ref{key-assumption-2}), builds upon the previous results, and resolves the even-genus case in Theorem \ref{GreenConjecture} and Theorem \ref{Geometric}. Finally, section \ref{Oddgenus} completes the proof of Theorem \ref{GreenConjecture} in the odd-genus case. \smallskip

% Acknowledgement
\noindent\textbf{Acknowledgements.} We are grateful to Michael Kemeny for his generous ideas and numerous insightful discussions. We also thank J. Rathmann, Ananth Shankar, and Ziquan Yang for valuable conversations regarding deformation spaces of K3 surfaces in positive characteristic.\smallskip

This research was supported by the Office of the Vice Chancellor for Research and Graduate Education at the University of Wisconsin–Madison, with funding from the Wisconsin Alumni Research Foundation. \smallskip

% Preliminaries
\section{Preliminaries}\label{Prelim}

% subsection_Lazarsfeld-Mukai Bundle
\subsection{Koszul Cohomology and Syzygies}
\begin{comment}
For any projective variety $X$, let $L$ be an ample line bundle on $X$. For any integer $p,q$, one defines $(p,q)$-th Koszul cohomology $K_{p,q}(X,L)$ as the cohomology group of the following sequence
$$\bigwedge^{p+1}H^0(X,L)\otimes H^0(X,L^{\otimes q-1})\to \bigwedge^{p}H^0(X,L)\otimes H^0(X,L^{\otimes q})\to \bigwedge^{p-1}H^0(X,L)\otimes H^0(X,L^{\otimes q+1}).$$
\end{comment}

Let $V$ be a finite-dimensional vector space over a field $k$, and let $S_V:=\bigoplus_{d\geq 0}\mathrm{Sym}^d(V)$ denote the symmetric algebra of $V$. Let $M=\oplus_{q\geq 0}M_q$ be a graded $S_V$-module. 
\begin{defn}
    The Koszul cohomology $K_{p,q}(M,V)$ is defined as the middle cohomology of the complex:
    $$\bigwedge^{p+1}V\otimes M_{q-1}\to\bigwedge^p V\otimes M_q \to\bigwedge^{p-1}V\otimes M_{q+1}.$$
\end{defn}

Koszul cohomology, introduced by Green, is a fundamental tool in the study of projective geometry. Let $X$ be a projective variety, $L$ a line bundle and $B$ a coherent sheaf on $X$. We define the graded $S_L:=S_{H^0(X,L)}$-module
$$\Gamma_X(B,L):=\bigoplus_{q\in\bbZ}H^0(X,L^{\otimes q}\otimes B).$$

We then define $K_{p,q}(X,B,L):=K_{p,q}(\Gamma_X(B,L),H^0(X,L))$. 

Given a subspace $V\subseteq H^0(X,L)$, we may regard $\Gamma_X(\cO_X,L)$ as a $S_W$-module and define $$K_{p,q}(X,B,L;V):=K_{p,q}(\Gamma_X(B,L),V).$$

When $B=\cO_X$, we simplify the notation and write $K_{p,q}(X,L):=K_{p,q}(X,B,L)$, and similarly $K_{p,q}(X,L;V)$. \smallskip
 
We say a polarized variety $(X,L)$ is \textit{projectively normal} if the natural map $\mathrm{Sym}^n\,H^0(X,L)\to H^0(X,L^{\otimes n})$ is surjective for all $n\geq 0$. If $X\subseteq\bbP_k^r$ is a projectively normal variety embedded by a very ample line bundle $L$, then its homogeneous coordinate ring $S_X:=S/I_X$ is a graded module over $S:=k[x_0,\cdots,x_r]$. The minimal free resolution of $S_X$ over $S$ gives rise to a graded Betti table $(\beta_{i,i+j})$, recording the number of generators in each homological degree and internal degree.

$$0\rightarrow\cdots\rightarrow \underbrace{\bigoplus_j S(-i-j)^{\beta_{i,i+j}}}_{\textrm{$i$-th syzygy}}\rightarrow\cdots\rightarrow  \underbrace{\bigoplus_j S(-1-j)^{\beta_{1,1+j}}}_{\textrm{1st syzygy}} \to S\to S_X\rightarrow 0$$

Note $\beta_{p,p+q}(S_X)=\mathrm{dim}_k\,\mathrm{Tor}_p(S_X,k)_{p+q}$. By the Koszul complex of $k$ and the isomorphism $S_X\simeq \Gamma_{X}(\mathcal{O}_X,L)$, it can also be computed by the following Koszul cohomology
$$\bigwedge^{p+1}H^0(X,L)\otimes H^0(X,L^{\otimes q-1})\to\bigwedge^p H^0(X,L)\otimes H^0(X,L^{\otimes q}) \to\bigwedge^{p-1}H^0(X,L)\otimes H^0(X,L^{\otimes q+1}). $$

We have the equality $\beta_{p,p+q}=\mathrm{dim}\,K_{p,q}(X,L)=:b_{p,q}$. We refer to \cite{EL12} for additional background and include below the necessary results that will be used.\smallskip

Let $L$ be a very ample line bundle on a variety $X$. Let $V\subseteq H^0(X,L)$ be a base-point free subspace of dimension $r$ that embeds $X\hookrightarrow\bbP^{r-1}$. The case when $V=H^0(X,L)$ corresponds to the complete linear series. The associated \textit{kernel bundle} $M_V$ is a vector bundle of rank $r-1$ defined by the following exact sequence
    $$0\to M_V\to V\otimes\cO_X\to L\to 0.$$

\begin{lem}[\cite{EL12}*{Prop 3.2}]\label{lem-kernel-bdl} Assume that $H^i(X,B\otimes\cO_X(mL))=0$ for all $i>0$ and $m\geq 0$. Then 
$$K_{i,j}(X,B,L;V)=H^1(X,\wedge^{i+1}M_V\otimes B\otimes L^{\otimes j-1}),\textrm{ for all }j\geq 1.$$ 
In particular, $K_{i,1}(X,L;V)=H^1(X,\wedge^{i+1}M_V)$ and $K_{i,2}(X,L;V)=H^1(\wedge^{i+1}M_V\otimes L)$.
\end{lem}

\begin{lem}[\cite{KM24}*{Prop 1.3}]\label{lem-linear-syzygy}
    Let $(X,L)$ be a projectively normal polarized variety. For any subspace $V\subseteq H^0(X,L)$, the natural map $K_{i,1}(X,L;V)\to K_{i,1}(X,L)$ is injective for all $i\geq 0$.
\end{lem}

% subsection_Green's Conjecture
\subsection{Green's Conjecture on Syzygies of Canonical Curves}\label{Greenconjecture}

For a non-hyperelliptic curve of genus $g$, its canonical bundle gives an embedding $C\hookrightarrow\bbP^{g-1}$. The nontrivial Koszul cohomology groups $K_{i,j}(C,\omega_C)$ appear in the first and second rows of the Betti table associated with the minimal free resolution of the canonical embedding.

\begin{center}
\begin{tabularx}{0.5\textwidth} { 
  | >{\centering\arraybackslash}X
  | >{\centering\arraybackslash}X 
  | >{\centering\arraybackslash}X 
  | >{\centering\arraybackslash}X
  | >{\centering\arraybackslash}X | }
\hline
 $b_{1,1}$  & $b_{2,1}$ & $\cdots$ & $b_{g-2,1}$ & $b_{g-3,1}$  \\ \hline
 $b_{1,2}$  & $b_{2,2}$ & $\cdots$ & $b_{g-2,2}$ & $b_{g-3,2}$  \\ \hline
\end{tabularx}
\end{center}
Here $b_{i,j}:=\dim\,K_{i,j}(C,\omega_C)$. The duality $K_{i,1}(C,\omega_C)\simeq K_{g-i-2,2}(C,\omega_C)^{\vee}$ implies the symmetry $b_{i,1}=b_{g-i-2,2}$. As a result, the first and second rows of the Betti table determine one another.\smallskip

Formulated in the Introduction, Green’s Conjecture on syzygies of canonical curves asserts that the existence of special linear series on an algebraic curve, captured by the Clifford index, can be detected from the syzygies of its canonical embedding. The Clifford index of $C$ is defined as follows
$$\mathrm{Cliff}(C):=\min\{\mathrm{deg}(A)-2h^0(A)+2\,|\,\textrm{for all special line bundle $A$ on $C$}\}.$$
where special line bundle $A$ means $h^0(C,A)>0$ and $h^1(C,A)>0$. \smallskip

For a general smooth curve of genus $g$, the sharpness of the Brill-Noether theorem implies its Clifford index $\mathrm{Cliff}(C)=\lfloor\frac{g-1}{2}\rfloor$. Accordingly, Green's conjecture predicts that the Betti table of the canonical embedding of such curve has the following form.

\begin{itemize}
	\item For even genus $g=2k$,
	\begin{center}
	\begin{tabularx}{0.7\textwidth}{ 
  | >{\centering\arraybackslash}X
  | >{\centering\arraybackslash}X 
  | >{\centering\arraybackslash}X 
  | >{\centering\arraybackslash}X
  | >{\centering\arraybackslash}X
  | >{\centering\arraybackslash}X
  | >{\centering\arraybackslash}X
  | >{\centering\arraybackslash}X | }
\hline
 $b_{1,1}$  & $\cdots$ & $b_{k-2,1}$ & $b_{k-1,1}$ & $0$ & $0$ & $\cdots$ & $0$ \\ \hline
 $0$  & $\cdots$ & $0$ & $b_{k-1,2}$ & $b_{k,2}$ & $b_{k+1,2}$ & $\cdots$ & $b_{2k-3,2}$ \\ \hline
	\end{tabularx}
	\end{center}
	\item For odd genus $g=2k+1$,
	\begin{center}
\begin{tabularx}{0.7\textwidth}{ 
  | >{\centering\arraybackslash}X
  | >{\centering\arraybackslash}X 
  | >{\centering\arraybackslash}X 
  | >{\centering\arraybackslash}X
  | >{\centering\arraybackslash}X
  | >{\centering\arraybackslash}X
  | >{\centering\arraybackslash}X
  | >{\centering\arraybackslash}X | }
\hline
 $b_{1,1}$  & $\cdots$ & $b_{k-2,1}$ & $b_{k-1,1}$ & $0$ & $0$ & $\cdots$ & $0$ \\ \hline
 $0$  & $\cdots$ & $0$ & $0$ & $b_{k,2}$ & $b_{k+1,2}$ & $\cdots$ & $b_{2k-2,2}$ \\ \hline
\end{tabularx}
\end{center}

\end{itemize}

In both cases, a single vanishing 
$$K_{k,1}(C,\omega_C)=0$$
 suffices to prove Green's Conjecture for a general smooth curve.  \smallskip

When $g=2k$, $K_{k,1}(C,\omega_C)\simeq K_{k-2,2}(C,\omega_C)^{\vee}$ by duality. Moreover, the difference of the Betti numbers on the anti-diagonal is governed by the Hilbert function  \cite{FG17}*{\S IV.1}. In particular, $\dim\,K_{k-1,1}(C,\omega_C)-\dim\,K_{k-2,2}(C,\omega_C)={2k-1 \choose k-2}$. The above vanishing implies $$\dim\,K_{k-1,1}(C,\omega_C)={2k-1 \choose k-2}.$$

Following the framework of Voisin \cites{VC02, VC05}, we study curves lying on polarized K3 surfaces $(X,L)$ of genus $g$ under specific assumptions on the Picard groups. As mentioned in the Introduction, for $\Char(\bbF)\neq 2$, the hyperplane restriction theorem \cite{GM84} asserts $
K_{i,1}(C,\omega_C)=K_{i,1}(X,L)$ for all $i$, 
whenever $C$ is a smooth, connected hyperplane section of $L$ on $X$, noting that $L|_C=\omega_C$ in this setting. The result holds under the vanishing condition $H^1(X,qL)=0$ for all $q\geq 0$. While the Kodaira-Ramanujam vanishing theorem fails in general in positive characteristic, it is known to hold for an ample line bundle $L$ on K3 surfaces; see \cite{HD16}*{\S II. Prop 3.1}. \smallskip

In summary, generic Green's conjecture for a canonical curve reduces to a single vanishing condition on a polarized K3 surface $(X,L)$ of genus $g=2k$ or $g=2k+1$:
\begin{equation}\label{voisin-thm}
    K_{k,1}(X,L)=0.
\end{equation}
As mentioned in the Introduction, the even-genus case is implied by Theorem \ref{Geometric}. In Section \ref{Oddgenus}, we establish the vanishing directly for a K3 surface of odd genus.

\subsection{The Geometric Syzygy Conjecture}
Inspired by Green's conjecture, Schreyer and von Bothmer defined a notion of \textit{rank} that applies to an element in the linear syzygy space $K_{i,1}(X,L)$ for any polarized variety $(X,L)$.
\begin{defn}[\cite{BG07}]
    The \textit{rank} of a syzygy $\alpha\in K_{i,1}(X,L)$ is defined as the minimal dimension of the subspace $V\subseteq H^0(X,L)$ such that $\alpha\in K_{i,1}(X,L;V)$, that is, $\alpha$ is represented by an element in $\bigwedge^p V\otimes H^0(X,L)$.  (see Lemma \ref{lem-linear-syzygy}) 
\end{defn}
There is always a lower bound $\mathrm{rank}(\alpha) \geq i+1$ for any $\alpha\in K_{i,1}(X,L)$. Syzygies of minimal rank admit a geometric interpretation \cites{BG07, AN07}: when equality holds, $X$ lies on a rational normal scroll $X_{\alpha}\subseteq \bbP(H^0(X,L)^{\vee})$, where $X_{\alpha}$ is the \textit{syzygy scheme} associated with $\alpha$. In particular, a syzygy $\alpha\in K_{1,1}(C,\omega_C)$ of rank 2 yields linear determinantal equations defining a canonical curve $C$, and in particular a quadric of rank 4. Therefore, the Geometric Syzygy Conjecture \ref{GreenConjecture} may be viewed as a natural generalization of the classical theorem of Andreotti–Mayer–Arbarello–Harris on the ideal of a general canonical curve. \smallskip

In Theorem \ref{Geometric}, for curve $C$ of genus $g=2k$, it is shown that the last syzygy space $K_{k-1,1}(C,\omega_C)$ is generated by 1-dimensional subspaces, represented by a rank $k$ syzygy, in the form of 
$$K_{k-1,1}\big(C,\omega_C;H^0(\omega_C\otimes A^{-1})\big)$$
where $A\in W_{k+1}^1(C)$ and thus $h^0(C,\omega_C\otimes A^{-1})=h^1(A)=k$. \smallskip

The inclusions $H^0(\omega_C\otimes A^{-1})\hookrightarrow H^0(\omega_C)$ are determined by choices of sections in $H^0(A)$. Accordingly, the induced inclusions $K_{k-1,1}(C,\omega_C;H^0(\omega_C\otimes A^{-1}))\hookrightarrow K_{k-1,1}(C,\omega_C)$ are likewise parameterized by sections of $H^0(A)$.  Varying the choice of section yields different subspaces $K_{k-1,1}(C,\omega_C;H^0(\omega_C\otimes A^{-1}))$ of $K_{k-1,1}(C,\omega_C)$, each consisting of syzygies of rank $k$. To generate the entire space $K_{k-1,1}(C,\omega_C)$, one may need several such subspaces arising from different sections. However, since $K_{k-1,1}(C,\omega_C)$ is finite-dimensional, only finitely many such syzygies of rank $k$ are needed to span it.

% subsection_Brill-Noether Theory
\subsection{Brill-Noether Theory}\label{BrillNoether}

Classical Brill-Noether theory concerns special divisors on a smooth curve $C$ of genus $g$ over an algebraically closed field. For integers $r,d$, the Brill-Noether number is defined by $\rho(r,d,g)=g-(r+1)(g+r-d)$. The Brill-Noether locus is defined as $$W_d^r(C):=\{\textrm{line bundle $A$ on $C$}\,|\, \deg(A)=d, h^0(A)\geq r+1\}.$$ 
We also define $G_d^r(C):=\{(V,A)\,|\,r+1\textrm{ dimensional subspace }V\subseteq H^0(A), A\in W_d^r(C)\}$ which parametrizes linear systems of degree $d$ and dimension $r$ on $C$. Brill-Noether asserts the following for a general curve $C$, (i) if $\rho<0$, $W_d^r(C)$,  and hence $G_d^r(C)$, is empty; (ii) if $\rho\geq 0$, then $G_d^r(C)$, and therefore $W_d^r(C)$, is not empty. \smallskip

When working over $\bbC$, Lazarsfeld \cite{LR86} showed that, if $\Pic(X)=\bbZ[L]$, then every smooth curve $C\in|L|$ is Brill-Noether general. Moreover, a curve $C$ is said to satisfy Petri's condition if for every line bundle $A$ on $C$, the natural multiplication map $H^0(A)\otimes H^0(\omega_C\otimes A^{-1})\to H^0(\omega_C)$ is injective. Gieseker proved that a general curve satisfies Petri's condition, refining the Brill-Noether theory of the locus $W_d^r(C)$; see \cite{ACGH}*{\S IV}. In fact, Lazarsfeld \cite{LR86} also proved that a general smooth member of $C\in|L|$ is Petri general. His elegant proof relies essentially on the generic smoothness theorem in characteristic 0, which can be highly problematic in positive characteristic.\smallskip

By Ogus' Theorem \cite{OA79}*{Thm 2.9}, the geometric generic fiber of a versal deformation of polarized K3 surfaces over $\bbF$ in positive characteristic has Picard number one. Hence we may assume $(X,L)$ is a polarized K3 surface over \textit{some} algebraically closed field containing $\bbF$ such that $\Pic(X)=\bbZ[L]$. Following an argument similar to \cite{LR86}*{Cor 1.3}, any Lazarsfeld-Mukai bundle associated to a prescribed pair $(C,A)$ is simple. Consequently, its Chern characters $\chi(E\otimes E^{\vee})=2-2\rho(A)=2h^0(E\otimes E^{\vee})-h^1(E\otimes E^{\vee})\leq 2$ which implies $\rho(A)\geq 0$. In fact, as shown in \cite{LR86}*{Cor 1.4}, one can further deduce that for every smooth member $C \in |L|$ and every line bundle $A$ on $C$, the inequality $\rho(A)\geq 0$ holds. That is, every smooth curve \( C \in |L| \) is Brill–Noether general.

\subsection{Lazarsfeld-Mukai Bundle}\label{LMbundle}
In \cite{LR86}, Lazarsfeld first introduced a vector bundle $E=E_{C,A}$ on a K3 surface $X$ associated with a pair $(C,A)$, where $C$ is a smooth curve on $X$ and $A$ is a base-point free line bundle on $C$ with $h^0(A)=r+1$ and $\deg\,A=d$. As mentioned in the Introduction, $E^{\vee}$ fits into an exact sequence
$$0\to E^{\vee}\to H^0(C,A)\otimes \cO_X\to i_*A\to 0,$$
where $i:C\hookrightarrow X$ is the inclusion. Dualizing it gives:
\begin{equation}\label{Dual}
0\to H^0(C,A)^{\vee}\otimes\cO_X\to E\to i_*(\omega_C\otimes A^{-1})\to 0.
\end{equation}
As computed in \cite{LR86}, $E$ is a vector bundle of rank $r+1$, and $c_1(E)=\cO_X(C)$, $c_2(E)=d$; Furthermore, $h^0(E)=h^0(A)+h^1(A)$, $h^1(E)=h^2(E)=0$, $h^0(E^{\vee})=h^1(E^{\vee})=0$. Moreover $\chi(E\otimes E^{\vee})=2h^0(E\otimes E^{\vee})-h^1(E\otimes E^{\vee})=2-2\rho(A)$ where $\rho(A)=g(C)-h^0(A)\cdot h^1(A)$. \smallskip

This paper focuses on the even-genus case when $g=2k$, $d=k+1$ and $r=1$, and $C\in|L|$ for an ample line bundle $L$ on $X$. Hence the Lazarsfeld-Mukai bundle $E=E_{C,A}$ is of rank 2. By the relation between Chern classes and the degeneracy locus of sections of a vector bundle, a general section $s\in H^0(E)$ gives a zero-dimensional $Z(s)$ on a general curve $C\in|L|$ with $\deg Z(s)=k+1$. Thus, we obtain the exact sequence
\begin{equation}\label{short-exact-vector-bdle}
    0\to\cO_X\xrightarrow{s}E\xrightarrow{\bigwedge s}L\otimes I_{Z(s)}\to 0.
\end{equation}
Note that $H^1(X,L\otimes I_{Z(s)})\simeq H^2(X,\cO_X)$ has 1-dimension.\smallskip

Now take another section $s'\in H^0(E)$ linearly independent from $s$, and assume that $s, s'$ vanishes exactly along the curve $C$. Then the exact sequence above can be refined to:
\begin{equation}\label{Dual2}
    0\to \langle s,s'\rangle\otimes\cO_X\to E\to i_*\big(\omega_C\otimes\cO_C(-Z(s))\big)\to 0.  
\end{equation}

Denote $\cO_C(Z(s))$ by $A$. The exact sequence above implies $h^0(\omega_C\otimes A^{-1})=h^1(A)=k$. Dualizing (\ref{Dual2}) yields:
$$0\to E^{\vee}\to \langle s,s'\rangle^{\vee}\otimes\cO_X\to i_*A\to 0.$$
In particular, $\deg(A)=c_2(E)=k+1$. By the Riemann-Roch theorem, this forces $h^0(A)=2$. Comparing (\ref{Dual}) and (\ref{Dual2}), we can identify $H^0(C,A)$ with $\langle s,s'\rangle^{\vee}$. Therefore, $A$ is base-point free. In summary, any nonzero regular section $s\in H^0(E)$ gives rise to a base-point free line bundle $A\in W_{k+1}^1(C)$ on some curve $C\in|L|$ that contains $Z(s)$.  \smallskip

\subsection{Symmetric, divided and exterior powers}
We refer to \cite{AFPRW}*{\S 3.1} for an algebraic setup over a positive characteristic field $\bbF$ with $\Char(\bbF)=p$. For any $n$-dimensional vector space $E$ over $\bbF$ with a basis $\{x_1,\cdots,x_n\}$, we have a natural isomorphism 
\begin{equation}\label{eqn-iso}
    (\mathrm{Sym}^d E)^{\vee}\simeq \mathrm{Sym}^d E^{\vee}, \; \;  \text{for $1\leq d\leq p-1$ }.    
\end{equation}
Indeed, ${(\mathrm{Sym}^d E^{\vee})}^{\vee}$ is naturally isomorphic to $d$-th divided power of $E$:
$$D^d E:=\{ x \in E^{\otimes d} \; | \; \sigma(x)=x, \; \text{for all $\sigma \in \mathfrak{S}_d$} \}$$
via a natural pairing. $D^d E$ is generated by \textit{divided power monomials} $$x_1^{(a_1)}\cdots x_n^{(a_n)}:=\sum_{\sigma\in\mathfrak{S}_d} \sigma\cdot x_1^{\otimes a_1}\otimes\cdots\otimes x_n^{\otimes a_n}\textrm{ with }\sum_{i=1}^n a_i=d.$$
and the natural map
$$D^d E \to \mathrm{Sym}^d\,E,\, x_1^{(a_1)}\cdots x_n^{(a_n)}\mapsto \binom{d}{a_1,\,\cdots,\,a_n}x_1^{a_1}\cdots x_n^{a_n},\textrm{ where }\binom{d}{a_1,\,\cdots,\,a_n}=\frac{d!}{a_1!\cdots a_n!}$$ 
is an isomorphism when all the multinomial coefficients are invertible, which holds true for every $1\leq d\leq p-1$. \smallskip

Moreover, let $0\to E_1\to E_2\to E_3\to 0$ be an exact sequence of vector spaces over $\bbF$. Provided $d \leq p-1$, there are two long exact sequences, \cite{ABW82}*{\S V}:
\begin{equation}\label{long-exact-1}
    \cdots\to\bigwedge^{d-2}E_2\otimes\mathrm{Sym}^2 E_1\to\bigwedge^{d-1}E_2\otimes E_1\to\bigwedge^d E_2\to\bigwedge^d E_3\to 0,
\end{equation}
\vspace{-0.7cm}
\begin{equation}\label{long-exact-2}
    \cdots\to\mathrm{Sym}^{d-2}E_2\otimes\bigwedge^2 E_1\to\mathrm{Sym}^{d-1}E_2\otimes E_1\to\mathrm{Sym}^d E_2\to\mathrm{Sym}^d E_3\to 0.
\end{equation}
Note that the first exact sequence above involves the isomorphism from divided power to symmetric products, which requires the assumption $d \leq p-1$ (the second exact sequence remains valid without this restriction).
Further, provided $d \leq p-1$, we may dualize the sequence $0\to E_1\to E_2\to E_3\to 0$ to obtain the second exact sequence. Dualizing once more yields the following exact sequence:
\begin{equation}\label{long-exact-3}
    0 \to \mathrm{Sym}^dE_1 \to \mathrm{Sym}^d E_2 \to \mathrm{Sym}^{d-1} E_2 \otimes E_3 \to \mathrm{Sym}^{d-2} E_2 \otimes \bigwedge^2 E_3 \to \cdots
\end{equation}
It is not hard to see that all those (\ref{eqn-iso}), (\ref{long-exact-1}),  (\ref{long-exact-2}), and (\ref{long-exact-3}) extend naturally to vector bundles over a smooth projective variety over $\bbF$, provided $d\leq p-1$. \smallskip

\subsection{Key Assumption (\ref{key-assumption}) and (\ref{key-assumption-2})}\label{sec-key-assump}

Let $(X,L)$ be a K3 surface over $\bbF$ of genus $g=2k$. Suppose there exists a smooth curve $C\in|L|$ and a base-point free line bundle $A\in W_{k+1}^1(C)$ with $h^0(A)=2$. Let $E$ be a rank 2 Lazarsfeld-Mukai bundle associated with the pair $(C,A)$. Define $Z\subseteq X\times\bbP(H^0(E))$ to be the reduced closed subvariety $Z:=\{(x,s)\,|\,s(x)=0\}$. In this section, we prove Proposition \ref{Kem} under the assumption (\ref{key-assumption}): any nonzero $s\in H^0(E)$ is regular. As explained in the Introduction, (\ref{key-assumption}) is equivalent to (\ref{key-assumption-2}): the second projection $Z\to\bbP\big(H^0(E)\big)$ is finite and flat.

\begin{proof}[The proof of Proposition \ref{Kem}]
Under the key assumption (\ref{key-assumption}) or (\ref{key-assumption-2}), the following part of the argument proceeds verbatim as in \cite{KM20}. For the sake of completeness, we briefly summarize it here. \smallskip

Set $\bbP:=\bbP(H^0(E))\simeq \bbP^{k+1}$. Consider $X\times \bbP$ with projections $p:X\times \bbP\to X$ and $q:X\times \bbP\to \bbP$. Define the incidence subvariety $Z:= \{ (x, s) \in X \times \mathbb{P} \mid s(x) = 0 \}$. Since $E$ is globally generated, then $p|_Z:Z\to X$ is a projective bundle, and hence smooth. \smallskip

Moreover, $q|_{Z}:Z\to\bbP$ is finite and flat by assumption (\ref{key-assumption-2}). We have \smallskip
\begin{equation}\label{eqn-vector-bdl}
    0\to\cO_X\boxtimes \cO_{\bbP}(-2)\to E\boxtimes\cO_{\bbP}(-1)\to p^*L\otimes I_Z\to 0.
\end{equation}
Note that this is a universal version of (\ref{short-exact-vector-bdle}) on $X\times\bbP$. \smallskip

Let $M_L$ be the kernel bundle of $(X,L)$ and $\cM:=p^* M_L$ a vector bundle of rank $2k$ on $X\times\bbP$. Let $\pi:B\to X\times \bbP$ be the blow-up along $Z$ with exceptional divisor $D$. Note that $$H^1\Big(B,\bigwedge^{i}\pi^*\cM\Big)=H^1\Big(X\times\bbP,\bigwedge^i \cM\Big)\simeq H^1(X,\wedge^i M_L)$$ by the K\"{u}nneth formula, as $H^1(X,\cO_X)=0$. \smallskip

Set $p':=p\circ\pi, q':=q\circ\pi$. The sheaf 
$$\cW:=\mathrm{Coker}(q'_*(p'^*L\otimes I_D)\to q'_*p'^*L)\simeq \mathrm{Coker}(q_*(p^*L\otimes I_Z)\to q_*p^*L)$$
on $\bbP$ is locally free of rank $k$; see \cite{KM20}*{Lem 1}. This essentially follows from the assumption that every nonzero $s\in H^0(E)$ is regular. \smallskip

We define two \textit{universal secant bundles} $\cS$ and $\Gamma$ sitting in the first column of the following diagram. They are both vector bundles on $B$ of rank $k$.
\[
\begin{tikzcd}
            & 0 \arrow[d]           & 0 \arrow[d]           & 0 \arrow[d]           &  \\
0 \arrow[r] & \cS \arrow[r] \arrow[d] & q'^*q'_*(p'^*L\otimes I_D) \arrow[r] \arrow[d] & p'^*L\otimes I_D \arrow[r] \arrow[d] & 0  \\
0 \arrow[r] & \pi^*\cM \arrow[r] \arrow[d] & q'^*q'_*p'^*L \arrow[r] \arrow[d] & p'^*L \arrow[r] \arrow[d] & 0 \\
0 \arrow[r] & \Gamma \arrow[r] \arrow[d] & q'^*\cW \arrow[r] \arrow[d] & p'^*L|_D \arrow[r] \arrow[d] & 0 \\
            & 0                     & 0                     & 0                     &  
\end{tikzcd}
\]
As $\mathrm{rank}\,\Gamma=k$, $\bigwedge^{k+1}\Gamma =0$, the first column yields a long exact sequence; see (\ref{long-exact-1}), provided $k+1\leq p-1$, equivalently $p\geq k+2=\frac{g+4}{2}$. 
$$    \cdots\to\bigwedge^{k-1}\pi^*\cM\otimes\mathrm{Sym}^2 \cS\to\bigwedge^{k}\pi^*\cM\otimes \cS\to\bigwedge^{k+1} \pi^*\cM\to 0.$$
To prove the desired vanishing $K_{k,1}(X,L)=H^1(X,\wedge^{k+1} M_L)=H^1(B,\bigwedge^{k+1}\pi^*\cM)=0$, it suffices to show
$$H^i\Big(B,\bigwedge^{k+1-i}\pi^*\cM\otimes\mathrm{Sym}^i\cS\Big)=0,\textrm{ for }1\leq i\leq k+1.$$
Note that the above conditions are established in \cite{KM20}*{Thm 4} after careful computation. The proof relied on an injection $\mathrm{Sym}^{k+1}\cF\to\mathrm{Sym}^{k}\cF\otimes\cF$ for a vector bundle $\cF$. This holds true if $k+1\leq p-1$.  Therefore, generic Green's conjecture for even genus holds, provided $p\geq\frac{g+4}{2}=k+2$ and key assumption (\ref{key-assumption}) holds. We proceed to prove Proposition \ref{Kem} under the same assumption. \smallskip

We have another long exact sequence from; see (\ref{long-exact-1}), provided $k\leq p-1$, 
\begin{equation*}
    \cdots\to\bigwedge^{k-1}\pi^*\cM\otimes \cS \otimes p'^*L\to\bigwedge^{k} \pi^*\cM\otimes p'^*L\to \bigwedge^k\Gamma\otimes p'^*L\to 0.    
\end{equation*}
One computes $\bigwedge^k\Gamma\xrightarrow{\simeq} I_D\otimes q'^*\cO_{\bbP}(k)$; see \cite{KM20}*{Lem 5}. Pushing-forward by $\pi:B\to X\times \bbP$ in the above last two terms yields a natural morphism of sheaves on $X\times\bbP$:
\begin{equation}\label{eqn-natural-morphism}
    \bigwedge^k\cM\otimes p^*L\to \bigwedge^k\pi_*\Gamma\otimes p^*L \simeq L\boxtimes\cO_{\bbP}(k)\otimes I_Z.
\end{equation} 
Furthermore, applying direct functor $R^*q_*$ gives a natural morphism of sheaves on $\bbP$:
$$\psi:R^1q_*\Big(\bigwedge^k\cM\otimes p^*L\Big)\to R^1q_*\left(L\boxtimes\cO_{\bbP}(k)\otimes I_Z\right).$$
Now $R^1q_*\bigwedge^k\cM\otimes p^*L$ is a trivial vector bundle on $\bbP$ with fiber identified as 
$$H^0\Big(R^1q_*\big(\bigwedge^k\cM\otimes p^*L\big)\Big)=H^1\Big(X\times\bbP, \bigwedge^k\cM\otimes p^*L\Big)=H^1(X,\wedge^k M_L\otimes L)=K_{k-1,2}(X,L)=K_{k-1,1}(X,L)^{\vee}.$$
On the other hand, after tensoring (\ref{eqn-vector-bdl}) by $\cO_{\bbP}(k)$, the long exact sequence induced by direct pushing-forward of $q:X\times\bbP\to\bbP$ splits into
$$0\to\cO_{\bbP}(k-2)\to H^0(E)\otimes\cO_{\bbP}(k-1)\to q_*(L\boxtimes\cO_{\bbP}(k)\otimes I_Z)\to 0,$$
$$0\to R^1q_*(L\boxtimes\cO_{\bbP}(k)\otimes I_Z)\to R^2q_*\cO_{X\times\bbP}\otimes\cO_{\bbP}(k-2)\to 0.$$
Since $R^2q_*\cO_{X\times\bbP}\otimes\cO_{\bbP}(k-2)\simeq \cO_{\bbP}(k-2)$, $\psi$ naturally identified with 
$$\psi:K_{k-1,1}(X,L)^{\vee}\otimes\cO_{\bbP}\to\cO_{\bbP}(k-2).$$
Taking the global section induces a surjection $K_{k-1,1}(X,L)^{\vee}\to H^0(\bbP,\cO_{\bbP}(k-2))=\mathrm{Sym}^{k-2}\,H^0(E)^{\vee}.$ Dualizing it yields an injection $\mathrm{Sym}^{k-2} H^0(E) \to K_{k-1,1}(X,L)$, as $\mathrm{Sym}^{k-2} H^0(E)\simeq (\mathrm{Sym}^{k-2}\,H^0(E)^{\vee})^{\vee}$ holds under the condition $p\geq k+2$. The injection becomes an isomorphism, as the vanishing $K_{k,1}(X,L)=0$ implies $\dim\,K_{k-1,1}(X,L)={2k-1 \choose k-2}=\dim\,\mathrm{Sym}^{k-2}H^0(E)$; see Section \ref{Greenconjecture}.\smallskip

As $\psi$ is an isomorphism on global sections and $\cO_{\bbP}(k-2)$ is globally generated, it must be surjective as a morphism of sheaves. Then $\psi$ induces a Veronese morphism of degree $k-2$
$$\widetilde{\psi}:\bbP\big(H^0(E)\big)\to\bbP\big(K_{k-1,1}(X,L)\big).$$
The fiber of $\psi$ over $[s]\in\bbP$ is identified with the natural morphism induced from (\ref{eqn-natural-morphism})
\[
\begin{tikzcd}
H^1(B_s,\wedge^k\pi_s^*M_L(L)) \arrow[r] \arrow[d, equal] & H^1(B_s,\wedge^k\Gamma_s(\pi_s^*L))  \arrow[d, equal]  \\
H^1(X,\wedge^k M_L(L)) \arrow[r]                             & H^1(X,\wedge^k {\pi_s}_*\Gamma(L))                                  
\end{tikzcd}
\]
where $\pi_s:B_s\to X\times\{[s]\}\simeq X$ is the blow-up of $X$ at the 0-dimensional locus $Z(s)$ with exceptional divisor $D(s)$. Consider the short exact sequences of vector bundles on $B_s$:
$$0\to \cS_s\to\pi_s^* M_L\to\Gamma_s\to 0,\;\textrm{ and }\;0\to\cS_s\to H^0\big(\pi_s^*L(-D(s))\big)\otimes \cO_{B_s}\to \pi_s^*L\big(-D(s)\big)\to 0.$$
Then $\mathrm{det}(\pi_s^*M_L)=\pi_s^*L^{-1}$, $\wedge^k \pi_s^*M_L(L)\simeq \wedge^k\pi_s^*M_L^{\vee}$ and $\wedge^k\Gamma_s(\pi_s^*L)\simeq \wedge^k\cS_s^{\vee}$ as $\mathrm{rank}\,\cS_s=\mathrm{rank}\,\Gamma_s=k, \mathrm{rank}\,\pi_s^* M_L=2k$. Therefore, the fiber of $\psi$ identifies to the map
$$H^1(B_s,\wedge^k \pi_s^* M_L^{\vee})\twoheadrightarrow H^1(B_s,\wedge^k \cS_s^{\vee}),$$
which is Serre dual to
$$H^1(B_s,\wedge^k\cS_s(D(s))\hookrightarrow H^1(B_s,\wedge^k\pi_s^* M_L(D(s))).$$
By Lemma \ref{lem-kernel-bdl} on kernel bundle description of syzygies, the inclusion is identified with
\[\begin{tikzcd}
K_{k-1,1}\big(B_s,\cO(D(s)),\pi_s^*L(-D(s))\big) \arrow[r, hook] \arrow[d, equal] & K_{k-1,1}\big(B_s,\cO(D(s)),\pi_s^*L\big)  \arrow[d, equal]  \\
K_{k-1,1}\big(X,L;H^0(L\otimes I_{Z(s)})\big) \arrow[r, hook]                                & K_{k-1,1}(X,L).                    
\end{tikzcd}
\]
as $H^0(B_s,\pi_s^*L^{\otimes n}(D_s))\simeq H^0(X,L^{\otimes n})$ for any $n$. This is equivalent to saying 
$$\widetilde{\psi}([s])=\big[K_{k-1,1}\big(X,L;H^0(L\otimes I_{Z(s)})\big)\big]\in\bbP\big(K_{k-1,1}(X,L)\big).$$
\end{proof}
 
% Universal Construction
\section{Family of Lazarsfeld-Mukai Bundles}\label{Uni}
In applying the Lazarsfeld-Mukai bundle method, it is advantageous to work with a K3 surface of Picard number one. While it may not exist over the original field $\bbF$, Ogus' theorem suggests considering a smooth family of primitively polarized K3 surfaces and passing to the geometric generic fiber. Accordingly, to prove Theorem \ref{Main}, we construct a family of Lazarsfeld-Mukai bundles associated to a smooth family of primitively polarized K3 surfaces $(\cX,\cL)\to \cT$.  

\begin{defn}
A smooth family of primitively polarized K3 surfaces of degree $2g-2$, is a smooth projective morphism of $\bbF$-varieties $\pi:(\cX,\cL)\to \cT$ with an ample line bundle $\cL$ on $\cX$, such that for every $t\in \cT$, each fiber $(X_t,\cL_t)$ is a primitively polarized K3 surface of degree $2g-2$. 
\end{defn}

\begin{lem}\label{Curve}
Assume $\Char(\bbF)\neq 2$. Let $\pi:(\cX,\cL)\to \cT$ be a smooth family of primitively polarized K3 surfaces of degree $2g-2$. Possibly after replacing $T$ by a non-empty Zariski-open subset, there exists an embedded smooth family of curves
\[
\begin{tikzcd}
\cC \arrow[rd, "\pi|_{\cC}"'] \arrow[r, hook, "i"] & \cX \arrow[d, "\pi"] \\ & \cT          
\end{tikzcd}
\]
such that for every $t\in \cT$, $C_t\subseteq X_t$ is a smooth curve of genus $g$ with $\cO_{X_t}(C_t)\simeq \cL_t$. We may interchangeably use the notation $(\cX,\cC)\to \cT$ as a smooth family of primitively polarized K3 surfaces with degree $2g-2$.  
\end{lem}
\begin{proof}
By reducing to an irreducible component, we may assume $\cT$ to be irreducible. Since $\dim\,H^0(X_t,\cL_t)=g+1$ for every $t\in \cT$, it follows that $\pi_*\cL$ is a locally free sheaf on $\cT$ by Grauert's theorem. Consider the projective bundle $\bbP(\pi_*\cL)\to \cT$, there is a universal family
\[
\begin{tikzcd}
\cC \arrow[rd, "p|_{\cC}"'] \arrow[r, hook] & \cX\times_{\cT} \bbP(\pi_*\cL) \arrow[d, "p"] \\ & \bbP(\pi_*\cL)       
\end{tikzcd}
\quad \textrm{where } 
\cC:= \left\{ (x,[s]_t)\in \cX\times_{\cT}\bbP(\pi_*\cL)\,\Bigg|
\begin{array}{c}
x\in V([s]_t)\subseteq X_t  \\
\,[s]_t\in H^0(X_t,\cL_t) \textrm{ and }t\in \cT
\end{array}\right\}
\]
the fiber over $[s]_t\in\bbP(\pi_*\cL)$ along $p|_{\cC}$ is just the vanishing set of $[s]_t\in H^0(X_t,\cL_t)$ in $X_t$. Denote it as $C_{[s]_t}\subset X_t$, it is a curve of arithmetic genus $g$ (possibly singular). 

A general member $C_t\in|\cL_t|$ is smooth. After restricting to a non-empty Zariski open of $\bbP(\pi_*\cL)$, one may assume each fiber $C_{[s]_t}$ is smooth. Moreover, choose a local section of $\bbP(\pi_*\cL)\to \cT$ over some non-empty open $\cT'$ of $\cT$, i.e. $s:\cT'\to\bbP(\pi_*\cL)$. Pulling back the above diagram along $\cT'\to\bbP(\pi_*\cL)$ yields a family of curves $\cC'\to \cT'$ inside $\cX':=\cX|_{\cT'}$, which satisfies all the required properties. In particular, each fiber $C'_t$ is smooth and lies in $|\cL_t|$.
\[
\begin{tikzcd}
\cC' \arrow[rd, "\pi'|_{\cC'}"'] \arrow[r, hook, "i"] & \cX'\arrow[d, "\pi'"] \arrow[r,hook]& \cX \arrow[d,"\pi"]  \\ & \cT'\arrow[r,hook]& \cT          
\end{tikzcd}
\]
\end{proof}

\begin{lem}\label{Line}
Let $\pi:\cC\to \cT$ be a smooth family of genus $g$ curves. Assume the geometric generic fiber $C_{\taub}$ is Brill-Noether general. Let $r,d$ be such integers such that the Brill-Noether number $\rho(r,d,g)=0$. Then up to a base change to an \'{e}tale open of $\cT$, there exists a line bundle $\cA$ on $\cC$, such that for every $t\in \cT$, $\cA_t:=\cA|_{C_t}\in W_{d}^r(C_t)$ with  $h^0(C_t,\cA_t)=r+1$ and $\cA_t$ base-point free. 
\end{lem}

\begin{proof}
Without loss of generality, we may assume $\pi:\cC\to\cT$ admits a section, since any smooth family admits an \'{e}tale local section. Let $\mathcal{P}ic_{\cC/\cT}^d$ be the relative Picard scheme with degree $d$ associated to $\pi:\cC\to\cT$. The (relative) Brill-Noether variety $\cW_d^r(\pi)\subseteq \mathcal{P}ic_{\cC/\cT}^d$ associated to $\pi:\cC\to\cT$ is a $\cT$-scheme representing the functor (see \cite{ACGH2}*{\S XXI.3})
\[\begin{array}{ccl}
   \textsf{Sch}/\cT  & \to & \textsf{Sets}\\
   \cS  & \mapsto & \{[\cL]\in\mathcal{P}ic^d_{\cC/\cT}(\cS)\,|\,\textrm{the Fitting-rank}(R^1_{p_{\cS},*}\cL)\geq g-d+r \}
\end{array}
\]
where $p_{\cS}:\cC\times_{\cT}\cS\to\cS$ is the projection.
More precisely, set-theoretically, 
$$\supp(\cW_d^r(\pi))=\{(t,L)\,|\, t\in\cT, L\in\Pic^d(C_t) \textrm{ such that }L\in W_d^r(C_t)\}.$$

Since $\rho\geq 0$ implies $W_d^r(C)\neq\varnothing$ for every curve $C$, the $\cT$-points $\cW_d^r(\pi)(\cT)\neq \varnothing$. Choose any $\cA\in \cW_d^r(\pi)(\cT)$, i.e. for every $t\in \cT$, $\cA_t\in W_d^r(C_t)$ i.e. $h^0(C_t,\cA_t)\geq r+1$. By assumption, the geometric generic fiber $C_{\taub}$ is Brill-Noether general with $\rho(r,d,g)=0$. \smallskip

\textit{Claim}. For any $A\in W_d^r(C_{\tau})$ over the generic fiber $C_{\tau}$, we have
\[
\left\{\begin{array}{l}
   h^0(A)=r+1,  \\
   A \textrm{ is base-point free} 
\end{array}\right. \qquad\textrm{in particular,}\qquad
\left\{\begin{array}{l}
   h^0(A_{\tau})=r+1,  \\
   A_{\tau} \textrm{ is base-point free} 
\end{array}\right.
\]
Indeed, denote $\overline{A}$ as the pullback of $A$ by $C_{\taub}\to C_{\tau}$. By flat base change, we have $H^0(C_{\taub},\overline{A})=H^0(C_{\tau},A)\otimes_{\kappa(\tau)} \overline{\kappa(\tau)}$. If $h^0(A)=h^0(\overline{A})\geq r+2$, then $W^{r+1}_d(C_{\taub})\neq\varnothing$, but it contradicts $\rho(r+1,d,g)<0$ as $C_{\taub}$ is Brill-Noether general. Moreover, as the field extension is faithfully flat, if $A$ is not base-point free, equivalently $A$ globally generated, so is $\overline{A}$. Take the base-point free part $\overline{A}_b$ (i.e. $\im(H^0(\overline{A})\otimes\cO_{C_{\taub}}\to \overline{A})$), whose degree is strictly less than that of $\overline{A}$ but with the same global sections. Hence $W_{d-p}^r(C_{\taub})\neq \varnothing$ for some $p>0$, but it contradicts $\rho(r,d-p,g)<0$.

Finally, by the upper semi-continuity of the function $t\mapsto h^0(C_t,\cA_t)$ and base-point freeness being an open condition, we may shrink the base again to get the required property: there exists a line bundle $\cA$ on $\cC$ such that for every $t\in \cT$, 
$\cA_t\in W_{d}^r(C_t)$ with  $h^0(C_t,\cA_t)=r+1$ and $\cA_t$ base-point free. 
\end{proof}

With the argument in the above lemma, we are ready to prove Proposition \ref{Save}.

\begin{proof}[Proof of Proposition \ref{Save}]
Consider a primitively polarized K3 surface $(X,L)$ of degree $2g-2$ over an algebraically closed field $\bbF$. Let $\pi:(\cX,\cL)\to\cT$ be a smooth versal $\bbF$-deformation of $(X,L)$, i.e. smooth family of primitively polarized K3 surfaces  with line bundle $\cL$ of (relative) degree $2g-2$. The geometric generic fiber $(X_{\taub},\cL_{\taub})$ has $\Pic(X_{\taub})=\bbZ[\cL_{\taub}]$. In particular, any smooth $C_{\taub}\in|\cL_{\taub}|$ is Brill-Noether general; see Section \ref{BrillNoether}. \smallskip

The proof is similar to that of Lemma \ref{Curve}, focusing on an open dense locus $\bbP(H^0(\pi_*\cL))_s$ whose points are those $(t,[s]_t)$ such that $[s]_t$ defines a smooth curve $C\in|\cL_t|$. There is a universal family
\[
\begin{tikzcd}
\cC \arrow[rd, "p|_{\cC}"'] \arrow[r, hook] & \cX_{\bbP}:=\cX\times_{\cT} \bbP(\pi_*\cL)_s \arrow[d, "p"] \\ & \bbP(\pi_*\cL)_s       
\end{tikzcd}\]
such that the fiber over the point $(t,[s]_t)\in\bbP(\pi_*\cL)_s$ along $p|_{\cC}$ is the smooth curve $C_{[s]_t}\subseteq X_t$ defined by section $[s]_t\in H^0(\cL_t)$. \smallskip

\textit{Claim}. For a general fiber $(X_t,\cL_t)$ over $\cT$ and a general smooth curve $C\in|\cL_t|$, i.e. a general fiber over $\bbP(\pi_*\cL)_s$, each $A\in W_d^r(C)$ has (i) $h^0(A)=r+1$ and (ii) $A$ base-point free. \smallskip

If (i) fails for a general fiber, consider the relative Brill-Noether variety $\cW_{d}^{r+1}(p|_{\cC})$ associated to the smooth family of curves $p|_{\cC}:\cC\to\bbP(\pi_*\cL)_s$, then $\mathrm{supp}(\cW_d^{r+1}(p|_{\cC}))$ is not supported on any divisor of $\bbP(\pi_*\cL)_s$, hence also contains the generic points of $\bbP(\pi_*\cL)_s$ (maybe reducible). But as the Brill-Noether number $\rho(r,d,g)=0$, by the claim in the proof of Lemma \ref{Line}, those generic fibers, which are all over the generic point $\tau\in\cT$, are geometrically Brill-Noether general. Hence they do not admit any line bundle with $h^0(A)\geq r+2$. A contradiction. \smallskip

Similarly, if (ii) fails for a general fiber, consider the relative Brill-Noether variety $\cW_{d-i}^{r}(p|_{\cC})$ for $0<i<d$, then at least for some $i$, $\mathrm{supp}(\cW_{d-i}^{r}(p|_{\cC}))$ is not supported on any divisor of $\bbP(\pi_*\cL)_s$, hence also contains the generic points of $\bbP(\pi_*\cL)_s$. A contradiction by the same reason as above. \smallskip
\end{proof}

Let $\pi:(\cX,\cC)\to \cT$ be a smooth family of primitively polarized K3 surfaces of degree $2g-2$. Let $i:\cC\hookrightarrow \cX$ denote the closed embedding. Assume $\cA$ is a line bundle $\cA$ over $\cC$ such that $\cA_t\in W_{d}^r(C_t)$ and for every $t\in \cT$, $h^0(C_t,\cA_t)=r+1$ and $\cA_t$ base-point free. By Grauert's theorem, one easily sees that $\pi|_{C,*}\cA=\pi_*i_*\cA$ is a locally free sheaf on $\cT$ of rank $r+1$. We construct a family of Lazarfeld-Mukai bundles $\cE$ on $\cX$ associated with the line bundle $\cA$ on $\cC$ by the following steps. Take the kernel $\cF$ of the natural surjective evaluation map $\pi^*\pi_*i_*\cA \twoheadrightarrow i_*\cA$ to fit into an exact sequence
\begin{equation}\label{Lar}
0\to \cF\to \pi^*\pi_*i_*\cA \to i_*\cA\to 0.
\end{equation}

\begin{rmk}
The natural map $\pi^*\pi_*i_*\cA\to i_*\cA$ being surjective means $i_*\cA$ is globally generated along fibers, i.e. for any $t\in\cT$, $H^0(C_t,\cA_t)\otimes_{\kappa(t)} \cO_{X_t}\twoheadrightarrow i_{t,*}\cA_t$, which is true under our assumption.
\end{rmk}

\begin{prop}\label{UniLar}
Let $\pi:(\cX,\cC)\to \cT$ and $\cA$ be defined as above. Denote the associated line bundle $\cL=\cO_{\cX}(\cC)$. Then

\begin{enumerate}[\normalfont(i)]
    \item The kernel $\cF$ is a locally free sheaf of rank $r+1$ on $\cX$. Its dual $\cF^{\vee}$ is denoted as the family of Lazarfeld-Mukai bundles $\cE$ for $\pi:(\cX,\cC)\to \cT$ associated to the line bundle $\cA$ on $\cC$. We have $\det(\cE)=\cL$.
    \item For every $t\in \cT$, the following pullback of the exact sequence (\ref{Lar}) along the fiber $X_t$ is exact.
    \begin{equation}\label{LLar}
    0\to 
    F_t\to H^0(C_t,\cA_t)\otimes_{\kappa(t)} \cO_{X_t}\to i_{t,*}\cA_t\to 0  
    \end{equation}
    where $i_t:C_t\hookrightarrow X_t$ denotes the closed embedding. In particular, $\cE_t:=F_t^{\vee}$ is the Lazarsfeld-Mukai bundle associated to a pair $(C_t,\cA_t)$ on $X_t$. Moreover, it satisfies the following numerical invariants: $\det(\cE_t)=\cL_t=\cO_{X_t}(C_t), h^0(\cE_t)=h^0(\cA_t)+h^1(\cA_t),h^1(\cE_t)=h^2(\cE_t)=0$. 
\end{enumerate}
\end{prop}
\begin{proof}
For (i), note that $\cC\subseteq \cX$ is a divisor, $i_*\cA$ being locally isomorphic to $i_*\cO_{\cC}$, has homological dimension 1 over $\cO_{\cX}$. Hence $\cF$ is locally free. It is clear that $\cE=\cF^{\vee}$ has rank $r+1$. As the map $\cF\to\pi^*\pi_*i_*\cA$ drops rank along $\cC$, $\det(\cF)=\cO_{\cX}(-\cC)$, i.e. $\det(\cE)=\cL$. \smallskip

For (ii), we restrict the sequence (\ref{Lar}) to the fiber $j:X_t\hookrightarrow \cX$ over any point $t\in \cT$. It yields
\begin{equation*}
    0\to F_t \to \pi_t|_{C_t,*}\cA_t\otimes \kappa(t)\otimes_{\kappa(t)}\cO_{X_t}\to i_{t,*}\cA_t\to 0.
\end{equation*}
By the proper base change theorem, we have the natural isomorphism $\pi|_{C,*}\cA\otimes \kappa(t)\simeq H^0(C_t,\cA_t)$. Hence, we obtain the exact sequence (\ref{LLar}). By Serre duality, it gives
$$0\to H^0(C_t,\cA_t)^{\vee}\otimes_{\kappa(t)}\cO_{X_t}\to \cE_t\to i_{t,*}(\omega_{C_t}\otimes \cA_t^{\vee})\to 0$$
Then $h^0(\cE_t)=h^0(\cA_t)+h^1(\cA_t)$. Other numerical invariants are similarly computed as in \cite{LR86}.
\end{proof}
\smallskip
%Even genus curves
\section{Even Genus Curves: Green's Conjecture \& the Geometric Syzygy Conjecture}\label{Evengenus}

In this section, we use the family of Lazarsfeld-Mukai bundles constructed for a smooth versal deformation of K3 surfaces to prove our main Theorem \ref{Main}. This yields Green's Conjecture (Theorem \ref{GreenConjecture}) for general curves of even genus in char $p$ under the same characteristic bound. We then prove the Geometric Syzygy Conjecture for general curves of even genus in char $p$ under a different characteristic bound.  \smallskip

\begin{proof}[Proof of Theorem \ref{Main}]
Consider a primitively polarized K3 surface $(X,L)$ of genus $g=2k$ over $\bbF$. Let $\pi:(\cX,\cC)\to \cT$ be a smooth versal $\bbF$-deformation of $(X,L)$, such that the geometric generic fiber $(X_{\taub},\cL_{\taub})$ has $\Pic(X_{\taub})=\bbZ[\cL_{\taub}]$. Since the Brill-Noether number $\rho(1,k+1,2k)=0$, according to Lemma \ref{Line}, by passing to an \'{e}tale open of $\cT$, one may assume $\cA$ is a line bundle on $\cC$, flat over $\cT$, and that for every $t\in \cT$, $\cA_t\in W_{k+1}^1(C_t)$ with $h^0(C_t,\cA_t)=2$ and $\cA_t$ base-point free. Let $\cE$ be the family of rank 2 Lazarsfeld-Mukai bundles over $\cX$ associated to the line bundle $\cA$ on $\cC$. In particular, for every $t\in\cT$, $h^1(\cA_t)=k$ by the Riemann-Roch theorem, hence $h^0(\cE_t)=h^0(\cA_t)+h^1(\cA_t)=k+2$. By Grauert's theorem, $\pi_*\cE$ is a locally free sheaf of rank $k+2$ on $\cT$. \smallskip

Consider the projective bundle $\varphi:\bbP(\pi_*\cE)\to\cT$. Any point $p\in\bbP(\pi_*\cE)$ can be identified with $(t,[s]_t)$ for some $t\in\cT$ and $[s]_t\in\bbP(H^0(\cE_t))$. Take the fiber product with $\pi:\cX\to\cT$ to form a smooth family of K3 surfaces $\cX_{\bbP}:=\cX\times_{\cT}\bbP(\pi_*\cE)\to\bbP(\pi_*\cE)$. With a slight abuse of notation, we continue to denote the pullbacks of $\cL$ and $\cE$ to $\cX_{\bbP}$ by $\cL$ and $\cE$, respectively. Above any point $p=(t,[s]_t)\in\bbP(\pi_*\cE)$, the fiber of $\cX_p$ (canonically isomorphic as a scheme to $\cX_t$).  In particular, $(\cX_p,\cL_p,\cE_p)\simeq (X_t,\cL_t,\cE_t)$ as polarized K3 surfaces with bundles. Define the reduced closed subvariety $\cZ\subseteq \cX_{\bbP}$ on the closed set
$$\cZ=\{(x,t,[s]_t)\in \cX\times_{\cT}\bbP(\pi_*\cE)\,|\,x\in V([s]_t)\subseteq X_t\textrm{ for }(t,[s]_t)\in\bbP(\pi_*\cE)\}.$$

For any point $p=(t,[s]_t)\in\bbP(\pi_*\cE)$, we have $\cZ_p=\{x\in \cX_p\simeq X_t\,|\,s(x)=0\}$. On the other hand, for any point $t\in\cT$, denote by $\cZ_t$ the fiber of $\cZ$ over $t$, viewed as a subscheme of $X_t\times\bbP(H^0(\cE_t))$ the projection $\cZ\to\bbP(\pi_*\cE)$ with a natural projection $Z_t\to\bbP(H^0(\cE_t))$. Since the geometric generic fiber $(X_{\taub},\cL_{\taub})$ has Picard number one, by Remark \ref{Regular}, $\cZ_{\taub}\to\bbP(H^0(\cE_{\taub}))$ is finite and flat. As the field extension is faithfully flat, $\cZ_{\tau}\to\bbP(H^0(\cE_{\tau}))$ is also flat. \smallskip

By generic flatness theorem, we may choose a non-empty Zariski open set $\cU\hookrightarrow \bbP(\pi_*\cE)$ such that $\cZ|_{\cU}$ is flat over $\cU$. May assume $\cU$ contains the whole generic fiber $\bbP(H^0(\cE_{\tau}))\subset\bbP(\pi_*\cE)$ as $\cZ_{\tau}\to\bbP(H^0(\cE_{\tau}))$ is flat.
Since $\varphi:\bbP(\pi_*\cE)\to\cT$ is proper and surjective, the image of the complement of $\cU$ under $\varphi$ is a proper closed subset of $\cT$ avoiding the generic point $\tau\in\cT$. By possibly shrinking $\cU\subseteq\bbP(\pi_*\cE)$, we may take open $\cV\subseteq \cT$ such that $\varphi|_{\cU}:\varphi^{-1}(\cV)=\cU\subseteq\bbP(\pi_*\cE)\to \cV\subseteq \cT$ is surjective. Then one easily sees that for any $t\in \cV\subseteq\cT$, $\cZ_t\to\bbP(H^0(\cE_t))$ is flat and proper. Moreover, since over the generic point $\tau$, $\cZ_{\tau}\to\bbP(H^0(\cE_{\tau}))$ is a surjection between two projective varieties of the same dimension $k+1$, the map is quasi-finite. It follows that $\cZ_t\to\bbP(H^0(\cE_t))$ is also quasi-finite, hence finite, for any $t\in \cV$. \smallskip

Now it is clear that, for any closed point $t\in \cV\subseteq\cT$, $(X_t,\cL_t)$ satisfies the key assumption (\ref{key-assumption}) of Proposition \ref{Kem}. Shown in Section \ref{sec-key-assump}, it yields the desired result for such $(X_t,\cL_t)$. Therefore, a general primitively polarized K3 surface $(X,L)$ of genus $g=2k$ over $\bbF$ verifies Theorem \ref{Main}, provided $p\geq \frac{g+4}{2}$, and thus Theorem \ref{GreenConjecture} for genus $2k$ under the same characteristic assumption.
\end{proof}

The remaining section focuses on the Geometric Syzygy Conjecture for even genus curves.

\begin{prop}\label{Identify} 
Let $(X,L)$ be a general primitively polarized K3 surface of even genus $g=2k$. For a general smooth curve $C\in|L|$, let $A\in W_{k+1}^1(C)$ with $h^0(A)=2$ and base-point free. Let $E$ be the rank $2$ Lazarsfeld-Mukai bundle associated to such a pair $(C,A)$. We have a well-defined finite surjective map $d:\mathrm{Gr}_2(H^0(E))\to|L|$ with degree $\frac{1}{k+1}\binom{2k}{k}$. Moreover, the fiber $d^{-1}(C)$ over a general smooth curve $C\in|L|$ may be identified naturally with the Brill-Noether locus $W_{k+1}^1(C)$.
\end{prop}
\begin{proof}
The existence of such line bundle $A$ over a general smooth curve $C\in|L|$ is guaranteed by Proposition \ref{Save}. Taking the degree 2 wedge power of $H^0(E)\otimes\cO_X\to E$ gives $\bigwedge^2H^0(E)\otimes\cO_X \to \bigwedge^2 E=L$. Taking global sections yields $$\det:\bigwedge^2H^0(E)\to H^0(L).$$
A rational map $d:\mathrm{Gr}_2(H^0(E))\dasharrow|L|$ is given by taking a general two-dimensional subspace $W\subseteq H^0(E)$ to the degeneracy locus of the evaluation map $W\otimes\cO_X\xrightarrow{\textsf{ev}}E$, which is given by a section of $L$ via taking the degree 2 wedge power of $\textsf{ev}$
$$s:\cO_X\simeq\wedge^2 W\otimes\cO_X \xrightarrow{\wedge^2\textsf{ev}}\bigwedge^2 E=L.$$
To sum up, we have the following commutative diagram
\[
\begin{tikzcd}
\mathrm{Gr}_2(H^0(E)) \arrow[rr, "\textrm{Plücker embedding}", hook] \arrow[rd, "d"', dotted] &  &\bbP(\bigwedge^2 H^0(E)) \arrow[ld, "\bbP(\det)", dotted] \\   &  \,|L|      &
\end{tikzcd}
\]

Let $|L|_{gs}$ denote the open dense subset of $|L|$ contained by the locus of irreducible smooth curves. Consider a relative Brill-Noether variety $p:\cW_{k+1}^1(|L|_{gs})\to |L|_{gs}$. By Proposition \ref{Save}, we may further shrink $|L|_{gs}$ if necessary, to assume that, for any curve $C\in|L|_{gs}$, each $A\in W_{k+1}^1(C)$ has $h^0(C,A)=2$ and $A$ base-point free. Define a morphism $\psi:\cW_{k+1}^1(|L|_{gs})\to \mathrm{Gr}_2(H^0(E))$ as follows. Given $A\in W_{k+1}^1(C)$, we send it to $[H^0(C,A)^{\vee}]\in \mathrm{Gr}_2(H^0(E))$ via the following exact sequence
$$0\to H^0(C,A)^{\vee}\otimes\cO_X\xrightarrow{\textsf{ev}} E\to i_*(\omega_C\otimes A^{-1})\to 0.$$
The map $\textsf{ev}$ drops rank along $C$, hence the vanishing locus of the corresponding section $s:\cO_X\to\bigwedge^2 E=L$ is exactly $C$. That is to say, we have a commutative diagram
\[
\begin{tikzcd}
\cW_{k+1}^1(|L|_{gs}) \arrow[rr, "\psi" ] \arrow[rd, "p"'] &  & d^{-1}(|L|_{gs})\subseteq\mathrm{Gr}_2(H^0(E)) \arrow[ld, "d"] \\   &  \,|L|_{gs}  &   
\end{tikzcd}
\]
On the other hand, for any $[W]\in\mathrm{Gr}_2(H^0(E))$ such that $d(W)=C\in |L|_{gs}$, the cokernel is a rank one torsion free sheaf supported on $C$, i.e. a line bundle on $C$. We may write as the following where $A$ is some line bundle on $C$.
$$0\to W\otimes\cO_X \to E\to i_*(\omega_C\otimes A^{-1})\to 0.$$
This gives $h^0(\omega_C\otimes A^{-1})=k$. Dualizing it gives 
$$0\to E^{\vee}\to W^{\vee}\otimes \cO_X\to i_*A\to 0.$$
Hence $\deg(A)=c_2(E)=k+1$. By Riemann-Roch, $h^0(A)=2$. This identifies $W=H^0(C,A)^{\vee}$ and $A$ is base-point free, i.e. $A\in W_{k+1}^1(C)$. This process gives an inverse of $\psi$, which identifies $d^{-1}(C)\simeq W_{k+1}^1(C)$ for $C\in|L|_{gs}$. \smallskip

Let $\pi:(\cX,\cC)\to\cT$ be a versal $\bbF$-deformation of an arbitrary primitive polarized K3 surface. Over the geometric generic fiber, $\Pic(X_{\taub})=\bbZ[\cL_{\taub}]$. By Lemma \ref{Line} and Proposition \ref{UniLar}, we may further assume that there is a line bundle $\cA$ on $\cC$ such that for every $t\in\cT$, $\cA_t\in W_{k+1}^1(C_t)$ and $h^0(C_t,\cA_t)=2$ with $\cA_t$ base-point free. Let $\cE$ be the family of rank 2 Lazarsfeld-Mukai bundles associated with the pair $(\cC,\cA)$. Let $\pi^*\pi_*\cE\to\cE$ be the natural evaluation map; taking the degree 2 wedge power gives 
$$\bigwedge^2\pi^*\pi_*\cE=\pi^*\bigwedge^2\pi_*\cE\to\det\cE=\cL.$$

Hence after pushing forward to $\cT$ and by the projection formula, we have the determinant map of vector bundles over $\cT$
$$\det:\pi_*\pi^*\bigwedge^2\pi_*\cE=\pi_*\cO_{\cX}\otimes\bigwedge^2\pi_*\cE\to\pi_*\cL,$$
where $\pi_*\cO_{\cX}$ is a line bundle over $\cT$ by Grauert's theorem. Take the stalk at $t\in\cT$, the determinant map reduces to
$$\mathrm{det}|_t :\bigwedge^2H^0(\cE_t)\to H^0(\cL_t).$$

Consider the induced map between the associated projective bundles and a related Grassmannian bundle over $\cT$. Note that the projective bundle is insensitive to a twist by a line bundle, i.e. $\bbP(\cE)=\bbP(\cE\otimes\cL)$ for any line bundle $\cL$ and vector bundle $\cE$. We have the following commutative diagram of $\cT$-schemes
\[
\begin{tikzcd}
\mathrm{Gr}_2(\pi_*\cE) \arrow[rr, "\textrm{Plücker embedding}", hook] \arrow[rd, "d"', dotted] &  &\bbP(\bigwedge^2 \pi_*\cE) \arrow[ld, "\bbP(\det)", dotted] \\   &  \bbP(\pi_*\cL)  &   
\end{tikzcd}
\]
Over each point $t\in\cT$, the above diagram reduces to the following 
\[
\begin{tikzcd}
\mathrm{Gr}_2(H^0(\cE_t)) \arrow[rr, "\textrm{Plücker embedding}", hook] \arrow[rd, "d|_{X_t}"', dotted] &  &\bbP(\bigwedge^2 H^0(\cE_t)) \arrow[ld, "\bbP(\det|_{X_t})", dotted] \\   &  \,|\cL_t|      &
\end{tikzcd}
\]

Over the geometric generic fiber, $\Pic(X_{\taub})=\bbZ[\cL_{\taub}]$. It implies the determinant map $$\det|_{X_{\taub}}:\bigwedge^2H^0(\cE_{\taub})\to H^0(\cL_{\taub})$$
does not vanish on any element of rank 2. \cite{VC02}*{proof of eq. (3.18)}. Hence it defines a regular map $d|_{X_{\taub}}: \mathrm{Gr}_2(H^0(\cE_{\taub}))\to|\cL_{\taub}|$. It is clearly dominant. Moreover, it is finite since both are projective varieties with the same dimension $2k$. \smallskip

It follows from flat base change that over the generic fiber, we have a well-defined surjective regular map $d|_{X_{\tau}}: \mathrm{Gr}_2(H^0(\cE_{\tau}))\to|\cL_{\tau}|$, which is finite as well by the same reason as above. Note that there is a largest dense open $\cU$ of $\mathrm{Gr}_2(H^0(\pi_*\cE))$ where the rational map $d:\mathrm{Gr}_2(H^0(\pi_*\cE))\dasharrow \bbP(\pi_*\cL)$ is defined. Apparently, $\cU$ contains the whole generic fiber $\mathrm{Gr}_2(H^0(\cE_{\tau}))$. Since the Grassmannian bundle $\varphi:\mathrm{Gr}_2(\pi_*\cE)\to\cT$ is proper and surjective, the image of the complement of $\cU$ under $\varphi$ is a proper closed subset $\cS$ of $\cT$ avoiding the generic point $\tau$. Let $\cV=\cT\setminus \cS$ be a non-empty open, then $d$ is a regular map on $\varphi^{-1}(\cV)\subseteq \cU\subseteq\mathrm{Gr}_2(\pi_*\cE)$, i.e. over each fiber $t\in \cV\subseteq\cT$, $d|_{X_t}:\mathrm{Gr}_2(H^0(\cE_t))\to |\cL_t|$ is regular.
As the map $d|_{X_{\tau}}: \mathrm{Gr}_2(H^0(\cE_{\tau}))\to|\cL_{\tau}|$ over the generic fiber is finite, finiteness being an open condition, implies over a general fiber $X_t$, $d|_{X_t}:\mathrm{Gr}_2(H^0(\cE_t))\to |\cL_t|$ is finite. In particular, it implies $d|_{X_t}$ is surjective as well. Therefore, one may shrink $\cV$ if necessary to show that a general $(X,L)$ satisfies the prescribed property.\smallskip

Now we proceed to compute the degree of the finite surjective map $d:\mathrm{Gr}_2(H^0(E))\to |L|$ for a general $(X,L)$. To ease the notation, let $Z=\mathrm{Gr}(2,H^0(E))$ and $Y=|L|$. Note that the finite map $d:Z\to Y$ between two smooth schemes is automatically flat. Assume its degree is $n$. Now for any $y\in Y$, the fiber $d^{-1}(y)=Z_y$ is a finite subscheme of $Z$, and $\cO(Z_y)$ is a finite dimensional $\bbF$ vector space with $\dim_{\bbF}\,\cO(Z_y)=n$. \cite{LQ02}*{p.176} In fact, $Z_y$ is a complete intersection of hyperplane sections of  $\mathrm{Gr}(2,H^0(E))$ in the space $\bbP(\bigwedge^2 H^0(E))$ \cite{VC02}*{p.29}. Moreover, the number of hyperplane sections $\dim\,\bigwedge^2 H^2(E)-\dim\,H^0(L)=\binom{k+2}{2}-(2k+1)=\mathrm{codim}(X,\bbP(\bigwedge^2 H^0(E)))$ has the expected intersection dimension. Hence, the degree of the finite map $n=\dim_{\bbF}\,\cO(Z_y)$ coincides with the degree of the Grassmannian $\mathrm{Gr}(2,H^0(E))$ via its Plücker embedding. A formula for its degree can be found in \cite{Har92}*{p.247}
$$\deg\,\mathrm{Gr}(2,k+2)=(2k)!\prod_{i=0}^1\frac{i!}{(k+i)!}=\frac{(2k)!}{k!(k+1)!} =\frac{1}{k+1}\binom{2k}{k}.$$
This completes the proof.
\end{proof}

It is well-known that Theorem \ref{Main}, combined with an observation of Voisin, implies Theorem \ref{Geometric}, that is $K_{k-1,1}(C,\omega_C)$ is generated by syzygies of rank $k$ for general curves of even genus $g=2k$. (see the unpublished \cite{BV01}*{\S XI}).

\begin{proof}[Proof of Theorem \ref{Geometric}]
Let $(X,L)$ be a general primitively polarized K3 surface of genus $g=2k$, such that the regular map $d:\mathrm{Gr}(2,H^0(E))\to|L|$ is surjective and finite of degree $\frac{1}{k+1}\binom{2k}{k}$. By assumption $p>2k$, hence coprime to the degree $\frac{1}{k+1}\binom{2k}{k}$. It follows that the corresponding extension of function fields is separable. By generic smoothness for a generically separable map, for a general smooth $C\in|L|$, the fiber $d^{-1}(C)\simeq W_{k+1}^1(C)$ is reduced. Note that $A\in W_{k+1}^1(C)$ is mapped to $[H^0(A)^{\vee}]\in\mathrm{Gr}_2(H^0(E))$. \smallskip

By \cite{VC02}*{Prop 7}, the spaces $\mathrm{Sym}^{k-2}\,H^0(A)^{\vee}$, for $A\in W_{k+1}^1(C)$ generate $\mathrm{Sym}^{k-2}H^0(E)$. Each $H^0(A)^{\vee}$ corresponds to a line $T_A$ in $\bbP(H^0(E))$. A section $[t]_A\in H^0(A)^{\vee}\subseteq H^0(E)$ corresponds to a point $[t]_A$ sitting on the line $T_A$, whose image under 
$$\psi:\bbP(H^0(E))\to\bbP(K_{k-1,1}(X,L))$$
is $[K_{k-1,1}(X,L;H^0(L\otimes I_{Z([t]_A)}))]$; see Proposition \ref{Kem}. Let the set $$T:=\bigcup_{A\in W_{k+1}^1(C)}T_A$$
be the union of lines. The image of $T$ under $\psi$ is non-degenerate \cite{BV01}*{Thm 11.2}. This implies $K_{k-1,1}(X,L)$ is generated by the subspaces $K_{k-1,1}(X,L;H^0(L\otimes I_{Z(s)}))$, where $Z(s)$ given by a nonzero regular $s\in H^0(E)$ corresponds to a line bundle $A\in W_{k+1}^1(C)$ on a fixed curve $C$; see Section \ref{LMbundle}. After restricting to $C$, such spaces are identified with subspaces of the form
$$K_{k-1,1}\big(C,\omega_C;H^0(\omega_C\otimes A^{-1})\big)\subseteq K_{k-1,1}(C,\omega_C),\quad A\in W_{k+1}^1(C)$$
under the isomorphism $K_{k-1,1}(X,L)\simeq K_{k-1,1}(C,\omega_C)$. This completes the proof.
\end{proof} \smallskip

%Odd Genus Curves
\section{Odd Genus Curves: Green's Conjecture} \label{Oddgenus}
To complete the proof of Theorem \ref{GreenConjecture} in odd genus $g=2k+1\geq 5$, we adapt Kemeny's proof in \cite{KM20} through a similar deformation argument as in the even genus cases. We consider a versal deformation space $(\cX;\cL,\cR)\to\cT$ of K3 surfaces with a primitively polarized class of degree $2g-2$ and a rational class with intersection number 2 between them. We show shortly that the geometric generic fiber $X_{\taub}$ is ordinary with $\Pic(X_{\taub})=\bbZ[\cL_{\taub}]\oplus \bbZ[\cR_{\taub}]$ and $\cL_{\taub}^2=2g-2$, $\cR_{\taub}^2=-2$, $\cL_{\taub}\cdot \cR_{\taub}=2$.  \smallskip

\subsection{Deformation of Lattice Polarized K3 Surfaces}\label{sec-deform-of-lattice-K3}

Fix an algebraically closed field $\bbF$ with $\Char \bbF=p>2$. Let $A$ be an abelian surface over $\bbF$, $X$ be the associated Kummer surface, i.e. a minimal resolution $X\to A/\iota$ for the involution $\iota: A\to A, x\mapsto -x$. It is well-known $X$ is a K3 surface with an orthogonal decomposition $\NS(X)=\NS(A)\bigoplus\oplus_{i=0}^{15}\bbZ[E_i]$, where $E_i$ are the exceptional divisors with $E_i^2=-2$. For some primitive symmetric ample class $H$ on $A$ of degree $4f$ (i.e. $H\otimes \iota^* H$), it descends to $A/\iota$. Denote its pullback to $\zeta$ on $X$. Then $\zeta^2=2f$. Consider $$L=m\zeta-E_0-\sum_{i=1}^{i=15} n_i E_i$$
for some $m, n_i>0$. Note that $E_i\cdot E_j=0$ for $i\neq j$, $\zeta\cdot E_i=0$ for any $i$. $$L^2=2fm^2-2-2\sum_{i=1}^{15}n_i^2, \quad L\cdot E_0=2, \quad L\cdot E_i=2n_i, 1\leq i\leq 15.$$
Therefore, $L$ is ample if and only if $L^2=2d>0$. Observe that we can arrange the 17-tuple $(f,m,n_1,\ldots,n_{15})$ to let $L^2=2d$ for any $d>0$ by Lagrange's four-square theorem. Then we obtain an ample line bundle $L$ with degree $2d$ and a rational class $R:=E_0$, such that $L\cdot R=2$. \smallskip

For the above $(X;L,R)$, following the argument in \cite{LM11}*{Prop 3.1}, the universal $\bbF$-deformation $(\cX;\cL,\cR)\to\cT$ of $(X;L,R)$ has dimension 18 and the geometric generic fiber $X_{\taub}$ is ordinary, with $\Pic(X_{\taub})=\bbZ[\cL_{\taub}]\oplus \bbZ[\cR_{\taub}]$. Indeed, by Proposition 1.5 of \cite{DI81}, the locus in $\cS:=\mathrm{Spf}\, k[\![t_1,\cdots, t_{20}]\!]$ over which both $L$ and $R$ deform is $\Sigma(L)\cap \Sigma(R)$, where $\Sigma(L)$ (similar for $\Sigma(R)$) is the locus defined by a single equation, i.e. vanishing of the image $\pi:c_1(L)\to H^1(\cX,\Omega_{\cX/\cS}^1)$. Therefore, there exists a closed sub-scheme $\Spec A\subset \Spec k[\![t_1,\ldots,t_{20}]\!]$ of co-dimension 2 and a universal formal deformation of $(X;L,R)$ over $\mathrm{Spf} A$. Since $L$ is ample, this formal deformation is algebraizable by Grothendieck's existence theorem. Since the locus to parametrize supersingular K3 surfaces is of dimension $9<18$,   the general fiber $X_{\tau}$ is ordinary. By a similar argument in the proof of Ogus' Theorem, $\Pic(X_{\taub})=\bbZ[\cL_{\taub}]\oplus \bbZ[\cR_{\taub}]$. In particular, $\cL_{\taub}^2=2d, \cR_{\taub}^2=-2$ and $\cL_{\taub}\cdot \cR_{\taub}=2$.

\subsection{Vanishing Theorem for K3 surfaces} 
Suppose $X$ is a K3 surface over an algebraically closed field $\bbF$ with $\Char\,\bbF\neq 2$, such that $\Pic(X)=\bbZ[L]\oplus \bbZ[R]$ and $L^2=2g-2, R^2=-2, L\cdot R=2$ for $g=2k+1\geq 3$. Then both $L$ and $L':=L+R$ are base-point free, with $H^1(X,L)=H^1(X,L')=0$; see \cite{KM20}*{Lem 6} or \cite{VC05}*{Prop 1}. The main goal here is to prove
\begin{prop}\label{prop-odd-vanishing}
Let $(X,L,R)$ be as above. Assume $\Char\,\bbF=p\geq k+3=\frac{g+5}{2}$. Then $K_{k,1}(X,L)=0$.
\end{prop}
\begin{proof}
The proof is carried out similarly from Section \ref{sec-key-assump} as in \cite{KM20} and a recent simplified proof \cite{KM25}.  Generic Green's conjecture in odd-genus reduces to the vanishing 
$$K_{k,1}(X,L)^{\vee}=K_{k-1,2}(X,L)=H^1(X,\wedge^{k}M_L(L))=0.$$

Consider the exact sequence
\begin{equation}
    0\to\wedge^{k+1}M_L(L')\to\wedge^{k+1}M_{L'}(L')\to\wedge^k M_L(L)\to 0,
\end{equation}
Serre duality shows $H^2(\wedge^{k+1}M_L(L'))^{\vee}\simeq H^0(\wedge^k M_L(-R))\subseteq \wedge^k H^0(L)\otimes H^0(\cO_X(-R))=0$. Then $H^1(X,\wedge^k M_L(L))=0$ if and only if the following map is surjective: 
\begin{equation}\label{map-surj}
    H^1(X,\wedge^{k+1}M_L(L'))\to H^1(X,\wedge^{k+1}M_{L'}(L')).
\end{equation}

Let $\mu:X\to\widehat{X}$ be the map contracting the divisor $R$ to a singularity $v$ on a nodal K3 surface $\widehat{X}$. Indeed, the divisor $R$ in $X$ is exceptional while $v$ is a rational singularity on $\widehat{X}$ by Artin's Criterion \cite{AM62}*{Thm 1.7}. Therefore, $\mu:X\to\widehat{X}$ is a rational resolution of the singularity $v$. We have $\Pic(\widehat{X})\subseteq\Cl(\widehat{X})\simeq\Cl(X\setminus R)\simeq\bbZ$. As $\widehat{X}$ is projective \cite{AM62}*{Thm 2.3}, $\Pic(\widehat{X})$ must contain an ample class. Hence $\Pic(\widehat{X})\simeq\Cl(\widehat{X})$, i.e. $\widehat{X}$ is locally factorial. Moreover, $\widehat{X}$ is normal and Cohen-Macaulay by Serre's criterion. \cite{AM62}*{Thm 2.7} with projection formula implies the canonical divisor of $\widehat{X}$ is trivial and $H^1(\widehat{X},\mathcal{O}_{\widehat{X}})=0$, i.e. $\widehat{X}$ is a K3 surface with a single rational singularity at $v$. As $\Pic(\widehat{X})\simeq\bbZ$, $\widehat{X}$ admits a line bundle $\widehat{L}$ with $\mu^*(\widehat{L})=L'$ and $\widehat{L}^2=2(2k+2)-2=4k+2$.\smallskip

Similar construction in \ref{LMbundle} yields a rank two vector bundle $\widehat{E}$ on $\widehat{X}$ induced by a $g_{k+2}^1$ on a general curve $\widehat{C}\in|\widehat{L}|$. Set $E:=\mu^*\widehat{E}$, a rank two Lazarsfeld-Mukai bundle on $X$ induced by a $g_{k+2}^1$ on a general $C\in|L'|$. As $\mu_*E=\mu_*\mu^*\widehat{E}\simeq \widehat{E}$, $h^0(X,E)= h^0(\widehat{X},\widehat{E})=k+3$. Set $\bbP:=\bbP(H^0(X,E)), \widehat{\bbP}:=\bbP(H^0(\widehat{X},\widehat{E}))$. $\mu:X\to\widehat{X}$ induces an isomorphism $i:\bbP\to\widehat{\bbP}(\simeq\bbP^{k+2})$.\smallskip

Consider $Z\subseteq X\times\bbP$ the locus $\{(x,s)\,|\,s(x)=0\}$ and $\widehat{Z}\subseteq \widehat{X}\times \widehat{\bbP}$ the locus $\{(x,\widehat{s})\,|\,\widehat{s}(x)=0\}$. 
%Similar in (\ref{eqn-vector-bdl}), we have two exact sequences
%$$0\to\cO_X\boxtimes \cO_{\bbP}(-2)\to E\boxtimes\cO_{\bbP}(-1)\to p^*L'\otimes I_Z\to 0$$
%$$0\to\cO_{\widehat{X}}\boxtimes \cO_{\widehat{\bbP}}(-2)\to \widehat{E}\boxtimes\cO_{\widehat{\bbP}}(-1)\to p^*\widehat{L}\otimes I_{\widehat{Z}}\to 0$$
We run the same construction as in Section \ref{sec-key-assump}. Set $p:X\times\bbP\to X, q:X\times\bbP\to \bbP$, and $\widehat{p}:\widehat{X}\times\widehat{\bbP}\to \widehat{X}, \widehat{q}:\widehat{X}\times\widehat{\bbP}\to \widehat{\bbP}$. Let $\tau:=\mu\times i:X\times \bbP\to \widehat{X}\times\widehat{\bbP}$ be the resolution at $\{v\}\times\widehat{\bbP}$. \smallskip

Note that $\widehat{X}$ is a K3 surface with Picard number one. Then the second projection $\widehat{q}|_{\widehat{Z}}:\widehat{Z}\to \widehat{\bbP}$ is finite, similar to the condition (\ref{key-assumption-2}). Therefore, the sheaves $\widehat{\cW}':=\mathrm{Coker}(\widehat{q}_*(\widehat{p}^*\widehat{L}\otimes I_{\widehat{Z}})\to \widehat{q}_*\widehat{p}^*\widehat{L})$ and $\widehat{\cW}:=\mathrm{Coker}(\widehat{q}_*(\widehat{p}^*\widehat{L}\otimes I_{\widehat{Z}}\otimes \widehat{p}^*I_v)\to \widehat{q}_*\widehat{p}^*\widehat{L}\otimes \widehat{p}^*I_v)$ are locally free sheaves on $\widehat{\bbP}$ of rank $k+1$. Their pull-back by $i$ on $\bbP$ yields two locally free sheaves of rank $k+1$: $i^*\widehat{\cW}'=\cW':=\mathrm{Coker}(q_*(p^*L'\otimes I_Z)\to q_*p^*L')$ and $i^*\widehat{\cW}=\cW:=\mathrm{Coker}(q_*(p^*L\otimes I_Z)\to q_*p^*L)$; see \cite{KM25}*{Prop 1.4}. \smallskip

Let $\pi:B\to X\times\bbP$ be the blow-up along $Z$, with exceptional divisor $D$. Set $p':=p\circ\pi, q':=q\circ\pi$. Let $\cM_L=p^* M_L$ (respectively, $\cM_{L'}=p^* M_{L'}$) be the pull-back of the kernel bundle of $L$ (respectively $L'$). It has rank $2k+1$ (respectively $2k+2$). Using $\cW'$ and $\cW$, we construct two couples of universal secant bundles $(\cS_1,\Gamma_1)$ and $(\cS_2,\Gamma_2)$ sitting in a commutative diagram
\[
\begin{tikzcd}
0 \arrow[r] & \cS_2 \arrow[r] \arrow[d] & \pi^*\cM_L \arrow[r] \arrow[d] & \Gamma_2 \arrow[r] \arrow[d] & 0 \\
0 \arrow[r] & \cS_1 \arrow[r]           & \pi^*\cM_{L'} \arrow[r]           & \Gamma_1 \arrow[r]           & 0
\end{tikzcd}
\]
Note that $\mathrm{rank}\,\Gamma_1=\mathrm{rank}\,\Gamma_2=k+1$, with $\mathrm{det}\,\Gamma_1=\mathrm{det}\,\Gamma_2=q'^*\cO_{\bbP}(k+1)(D)$, $\mathrm{rank}\,S_1=k+1$, and $\mathrm{rank}\,S_2=k$. The vertical maps are induced from the natural inclusion $L=L'(-R)\hookrightarrow L'$. \smallskip

Under the characteristic bound $k+2\leq p-1$, equivalently $p\geq k+3=\frac{g+5}{2}$, similar in Section \ref{sec-key-assump}, we get a natural commutative diagram from the above,
\[
\begin{tikzcd}
H^1(B,\bigwedge^{k+1}\pi^*\cM_L\otimes p'^*L') \arrow[d] \arrow[r, "\psi'"] & H^1(B,\bigwedge^{k+1}\Gamma_2\otimes p'^*L') \arrow[d, "\simeq"] \\
H^1(B,\bigwedge^{k+1}\pi^*\cM_{L'}\otimes p'^*L') \arrow[r, "\psi"]           & H^1(B,\bigwedge^{k+1}\Gamma_1\otimes p'^*L')          
\end{tikzcd}
\]
Since $L'=\mu^*\widehat{L}$ has even genus $2k+2$, where $\widehat{X}$ is a nodal K3 surface $\widehat{X}$ with Picard number one. One may reduce to the even-genus case to show $\psi$ in the diagram is an isomorphism. By the K\"{u}nneth formula, as $H^1(X,\cO_X)=0$,
$$H^1\big(B,\bigwedge^{k+1}\pi^*\cM_L\otimes p'^*L'\big) = H^1(X,\wedge^{k+1}M_L(L')),\; H^1\big(B,\bigwedge^{k+1}\pi^*\cM_{L'}\otimes p'^*L'\big)=H^1(X,\wedge^{k+1}M_{L'}(L')).$$
Therefore, the surjectivity of (\ref{map-surj}) follows from the surjectivity of $\psi'$; see \cite{KM25}*{Thm 1.8}. The proof is complete.
\end{proof}

\begin{comment}
\begin{lem}
Let $X$ be such K3 surface as above. Assume $\Char\,\bbF=p\geq k+3 =\frac{g+5}{2}$. Then we have $K_{k+1,1}(X,L')=0$. In particular, there is a natural isomorphism
$$\mathrm{Sym}^{k-1}H^0(X,E)\simeq K_{k,1}(X,L').$$
\end{lem}
\begin{proof}
See \cite{KM20}*{Lem 8}. The reason we give the bound $p\geq k+3$ is because the proof involves exact sequences of symmetric product $\mathrm{Sym}^{k+2}$ of vector bundles in char $p$ field (see Section \ref{sec-key-assump}). 
\end{proof}

\begin{prop}
Let $X$ be such K3 surface as above over $\bbF$. Assume $\Char\,\bbF=p\geq k+3=\frac{g+5}{2}$, then $K_{k,1}(X,L)=0$.
\end{prop}
\begin{proof}
Similar argument as in the proof of Proposition \ref{Kem}, the whole computation carries out essentially the same in \cite{KM20}*{\S3}. The highest symmetric power of vector bundles \cite{KM20} uses is $\mathrm{Sym}^{k+1}$, meanwhile, the highest power of exterior product related to Sym is $\bigwedge^{k+2}$. Thus it requires a lower bound $p\geq k+3=\frac{g+5}{2}$.
\end{proof}

\end{comment}
\subsection{Green's Conjecture for Odd Genus $g=2k+1$}

Now we study a versal deformation space $(\cX;\cL,\cR)\to \cT$ introduced from the beginning of this section. Note that its geometric generic fiber $X_{\taub}$ has $\Pic(X_{\taub})=\bbZ[\cL_{\taub}]\oplus \bbZ[\cR_{\taub}]$ and $\cL_{\taub}^2=2g-2$, $\cR_{\taub}^2=-2$, $\cL_{\taub}\cdot \cR_{\taub}=2$.

\begin{prop}
Let $(X;L,R)$ be a general element of the above versal deformation $(\cX;\cL,\cR)\to\cT$. Assume $p\geq k+3=\frac{g+5}{2}$, we have $K_{k,1}(X,L)=0$.
\end{prop}

\begin{proof}
As shown in Section \ref{sec-deform-of-lattice-K3}, the geometric generic fiber $(X_{\taub},\cL_{\taub},\cR_{\taub})$ satisfies the assumption in Proposition \ref{prop-odd-vanishing}. Assume $p\geq k+3=\frac{g+5}{2}$, then $K_{k,1}(X_{\taub},\cL_{\taub})=0$. By the faithful flatness of field extension $\overline{k(\tau)}/k(\tau)$ and the upper semi-continuity of Koszul cohomology, we obtain $K_{k,1}(X,L)=0$ for a general member $(X;L,R)$ of the deformation space $\cT$.

\begin{comment}
Given the above data, we construct a family of rank 2 Lazarsfeld-Mukai bundles $\cE$ over $\cX\to\cT$ in the following way. \smallskip

By Lemma \ref{Curve}, up to a base change by some non-empty Zariski open of the base $\cT$, there exists an embedded smooth family of curves $\cC\subset \cX\to\cT$ associated to the line bundle $\cL':=\cL+ \cR$ whose geometric generic fiber $C_{\taub}$ is Brill-Noether general with genus $2k+2$. For any $t\in\cT$, $C_t\in |\cL_t'|=|\cL_t+\cR_t|$. Since $\rho(r,d,g)=\rho(1,k+2,2k+2)=0$, by Lemma \ref{Line}, up to a base change by an étale open of $\cT$, there exists a line bundle $\cA$ on $\cC$, such that for every $t\in\cT$, $\cA_t \in W_{k+2}^1(C_t)$ with $h^0(C_t,\cA_t)=2$ and $\cA_t$ base-point free. Following the same construction in the end of Section \ref{Uni}, let $\cE$ be a family of rank 2 Mukai-Lazarsfeld vector bundles associated to the family $(\cX,\cC)\to\cT$. Then over any $t\in\cT$, $\cE_t:=\cE|_{X_t}$ is the Mukai-Lazarsfeld vector bundle of rank 2 with $h^0(X_t,\cE_t)=k+3$, associated to $\cA_t$, a $g_{k+2}^1$ on $C_t\subset X_t$. \smallskip

Now,

we run lemma 5.1, 5.2 and Proposition 5.1 for the geometric generic fiber $(X_{\taub},\cL_{\taub},\cR_{\taub},\cE_{\taub})$. Assume $p\geq k+3=\frac{g+5}{2}$, then $K_{k,1}(X_{\taub},\cL_{\taub})=0$. By the faithful flatness of field extension $\overline{\bbF}/\bbF$ and upper semi-continuity of Koszul cohomology, for a general member $(X;L,R)$ of the deformation space $\cT$, we have $K_{k,1}(X,L)=0$.
\end{comment}
\end{proof}

Under the isomorphism $K_{k,1}(C,\omega_C)=K_{k,1}(X,L)=0$ for a general smooth $C\in |L|$ with genus $g=2k+1$, we prove generic Green's conjecture of odd genus in characteristic $p$ under the assumption $p\geq k+3=\frac{g+5}{2}$. Note that $g$ is odd, $p\geq \frac{g+5}{2}$ is equivalent to $p\geq \frac{g+4}{2}$. Therefore, we end up with the same characteristic bound only depending on the genus $g$ of the curve.

%%%%%%%%%%%%%%%%%%%%%%%%%%%%%%%%%%%%%%

\end{document}